\documentclass[reqno]{amsart}
\usepackage{etoolbox}

\makeatletter
\let\old@tocline\@tocline
\let\section@tocline\@tocline

\newcommand{\subsection@dotsep}{4.5}
\newcommand{\subsubsection@dotsep}{4.5}
\patchcmd{\@tocline}
  {\hfil}
  {\nobreak
     \leaders\hbox{$\m@th
        \mkern \subsection@dotsep mu\hbox{.}\mkern \subsection@dotsep mu$}\hfill
     \nobreak}{}{}
\let\subsection@tocline\@tocline
\let\@tocline\old@tocline

\patchcmd{\@tocline}
  {\hfil}
  {\nobreak
     \leaders\hbox{$\m@th
        \mkern \subsubsection@dotsep mu\hbox{.}\mkern \subsubsection@dotsep mu$}\hfill
     \nobreak}{}{}
\let\subsubsection@tocline\@tocline
\let\@tocline\old@tocline

\let\old@l@subsection\l@subsection
\let\old@l@subsubsection\l@subsubsection

\def\@tocwriteb#1#2#3{%
  \begingroup
    \@xp\def\csname #2@tocline\endcsname##1##2##3##4##5##6{%
      \ifnum##1>\c@tocdepth
      \else \sbox\z@{##5\let\indentlabel\@tochangmeasure##6}\fi}%
    \csname l@#2\endcsname{#1{\csname#2name\endcsname}{\@secnumber}{}}%
  \endgroup
  \addcontentsline{toc}{#2}%
    {\protect#1{\csname#2name\endcsname}{\@secnumber}{#3}}}%

\newlength{\@tocsectionindent}
\newlength{\@tocsubsectionindent}
\newlength{\@tocsubsubsectionindent}
\newlength{\@tocsectionnumwidth}
\newlength{\@tocsubsectionnumwidth}
\newlength{\@tocsubsubsectionnumwidth}
\newcommand{\settocsectionnumwidth}[1]{\setlength{\@tocsectionnumwidth}{#1}}
\newcommand{\settocsubsectionnumwidth}[1]{\setlength{\@tocsubsectionnumwidth}{#1}}
\newcommand{\settocsubsubsectionnumwidth}[1]{\setlength{\@tocsubsubsectionnumwidth}{#1}}
\newcommand{\settocsectionindent}[1]{\setlength{\@tocsectionindent}{#1}}
\newcommand{\settocsubsectionindent}[1]{\setlength{\@tocsubsectionindent}{#1}}
\newcommand{\settocsubsubsectionindent}[1]{\setlength{\@tocsubsubsectionindent}{#1}}

\renewcommand{\l@section}{\section@tocline{1}{\@tocsectionvskip}{\@tocsectionindent}{\@tocsectionnumwidth}{\@tocsectionformat}}%
\renewcommand{\l@subsection}{\subsection@tocline{1}{\@tocsubsectionvskip}{\@tocsubsectionindent}{\@tocsubsectionnumwidth}{\@tocsubsectionformat}}%
\renewcommand{\l@subsubsection}{\subsubsection@tocline{1}{\@tocsubsubsectionvskip}{\@tocsubsubsectionindent}{\@tocsubsubsectionnumwidth}{\@tocsubsubsectionformat}}%
\newcommand{\@tocsectionformat}{}
\newcommand{\@tocsubsectionformat}{}
\newcommand{\@tocsubsubsectionformat}{}
\expandafter\def\csname toc@1format\endcsname{\@tocsectionformat}
\expandafter\def\csname toc@2format\endcsname{\@tocsubsectionformat}
\expandafter\def\csname toc@3format\endcsname{\@tocsubsubsectionformat}
\newcommand{\settocsectionformat}[1]{\renewcommand{\@tocsectionformat}{#1}}
\newcommand{\settocsubsectionformat}[1]{\renewcommand{\@tocsubsectionformat}{#1}}
\newcommand{\settocsubsubsectionformat}[1]{\renewcommand{\@tocsubsubsectionformat}{#1}}
\newlength{\@tocsectionvskip}
\newcommand{\settocsectionvskip}[1]{\setlength{\@tocsectionvskip}{#1}}
\newlength{\@tocsubsectionvskip}
\newcommand{\settocsubsectionvskip}[1]{\setlength{\@tocsubsectionvskip}{#1}}
\newlength{\@tocsubsubsectionvskip}
\newcommand{\settocsubsubsectionvskip}[1]{\setlength{\@tocsubsubsectionvskip}{#1}}

\patchcmd{\tocsection}{\indentlabel}{\makebox[\@tocsectionnumwidth][l]}{}{}
\patchcmd{\tocsubsection}{\indentlabel}{\makebox[\@tocsubsectionnumwidth][l]}{}{}
\patchcmd{\tocsubsubsection}{\indentlabel}{\makebox[\@tocsubsubsectionnumwidth][l]}{}{}

\newcommand{\@sectypepnumformat}{}
\renewcommand{\contentsline}[1]{%
  \expandafter\let\expandafter\@sectypepnumformat\csname @toc#1pnumformat\endcsname%
  \csname l@#1\endcsname}
\newcommand{\@tocsectionpnumformat}{}
\newcommand{\@tocsubsectionpnumformat}{}
\newcommand{\@tocsubsubsectionpnumformat}{}
\newcommand{\setsectionpnumformat}[1]{\renewcommand{\@tocsectionpnumformat}{#1}}
\newcommand{\setsubsectionpnumformat}[1]{\renewcommand{\@tocsubsectionpnumformat}{#1}}
\newcommand{\setsubsubsectionpnumformat}[1]{\renewcommand{\@tocsubsubsectionpnumformat}{#1}}
\renewcommand{\@tocpagenum}[1]{%
  \hfill {\mdseries\@sectypepnumformat #1}}

\let\oldappendix\appendix
\renewcommand{\appendix}{%
  \leavevmode\oldappendix%
  \addtocontents{toc}{%
    \protect\settowidth{\protect\@tocsectionnumwidth}{\protect\@tocsectionformat\sectionname\space}%
    \protect\addtolength{\protect\@tocsectionnumwidth}{2em}}%
}
\makeatother

\makeatletter
\settocsectionnumwidth{2em}
\settocsubsectionnumwidth{2.5em}
\settocsubsubsectionnumwidth{3em}
\settocsectionindent{1pc}%
\settocsubsectionindent{\dimexpr\@tocsectionindent+\@tocsectionnumwidth}%
\settocsubsubsectionindent{\dimexpr\@tocsubsectionindent+\@tocsubsectionnumwidth}%
\makeatother

\settocsectionvskip{2pt}
\settocsubsectionvskip{0pt}
\settocsubsubsectionvskip{0pt}

\settocsectionformat{\bfseries}
\settocsubsectionformat{\mdseries}
\settocsubsubsectionformat{\mdseries}
\setsectionpnumformat{\bfseries}
\setsubsectionpnumformat{\mdseries}
\setsubsubsectionpnumformat{\mdseries}

\let\oldtableofcontents\tableofcontents
\renewcommand{\tableofcontents}{%
  \vspace*{-5\linespacing}
  \oldtableofcontents}

\makeatletter
\let\origsection=\section \def\section{\@ifstar{\origsection*}{\mysection}} 
\def\mysection{\@startsection{section}{1}\z@{.7\linespacing\@plus\linespacing}{.5\linespacing}{\normalfont\scshape\centering\S}}
\makeatother  

\usepackage{amssymb}
\usepackage{amsmath}
\usepackage{amsthm}
\usepackage{letterswitharrows}
\usepackage{mathrsfs}
\usepackage{graphicx}
\usepackage{latexsym}
\usepackage{stmaryrd}
\usepackage[shortlabels]{enumitem}
\usepackage{moreenum}
\usepackage[bookmarks,hyperfootnotes=false, psdextra=true]{hyperref}
\usepackage{multirow}
\usepackage{xcolor}
\usepackage{caption}
\usepackage{subcaption}

\colorlet{darkishRed}{red!60!black}
\colorlet{darkishBlue}{blue!60!black}
\colorlet{darkishGreen}{green!50!black}
\colorlet{darkerishGreen}{green!30!black}
\colorlet{lightishGreen}{green!70!black}
\hypersetup{
    draft = false,
    bookmarksopen=true,
    colorlinks,
    linkcolor={darkishBlue},
    citecolor={darkishBlue},
    urlcolor={darkishBlue}
}
\usepackage[nameinlink, capitalise, noabbrev]{cleveref}
\usepackage{thmtools}
\usepackage{nameref}
\crefformat{enumi}{#2#1#3}
\crefformat{equation}{#2(#1)#3}
\crefrangeformat{section}{Sections~#3#1#4--#5#2#6}
\crefname{mainresult}{Theorem}{Theorems}
\usepackage{doi}
\let\setminus=\smallsetminus
\usepackage{tikz}
\usepackage{tikz-cd}
\usetikzlibrary{calc,through,intersections,arrows, trees, positioning, decorations.pathmorphing, cd}
\usepackage{comment}
\usepackage{mathtools}
\usepackage{nccmath}
\usepackage{pifont}
\usepackage[utf8]{inputenc}
\usepackage[T1]{fontenc}
\usepackage{lmodern}
\usepackage[babel]{microtype}
\usepackage[english]{babel}
\usepackage{relsize}

\newcommand{\COMMENT}[1]{{}}

\usepackage{geometry}
\let\setminus=\smallsetminus
\renewcommand{\leq}{\leqslant}
\renewcommand{\geq}{\geqslant}

\let\eps=\varepsilon
\let\rho=\varrho
\let\phi=\varphi

\newcommand{\id}{\normalfont\text{id}}

\DeclareMathOperator{\Aut}{Aut}

\renewcommand{\subset}{\subseteq}

\newcommand{ \N } { \mathbb{N} }

\newcommand{\defn}[1]{{\color{darkishGreen}{\emph{#1}}}}
\newcommand{\mathdefn}[1]{{\color{darkishGreen}{{#1}}}}

\makeatletter

\def\calCommandfactory#1{%
   \expandafter\def\csname c#1\endcsname{\mathcal{#1}}}
\def\frakCommandfactory#1{%
   \expandafter\def\csname frak#1\endcsname{\mathfrak{#1}}}

\newcounter{ctr}
\loop
  \stepcounter{ctr}
  \edef\X{\@Alph\c@ctr}
  \expandafter\calCommandfactory\X
  \expandafter\frakCommandfactory\X
  \edef\Y{\@alph\c@ctr}
  \expandafter\frakCommandfactory\Y
\ifnum\thectr<26
\repeat

\newcommand{\bK}{\mathbf{K}}
\newcommand{\bT}{\mathbf{T}}
\newcommand{\bH}{\mathbf{H}}

\setenumerate{label={\normalfont (\roman*)}}

\newtheorem{theorem}{Theorem}[section] 
\newtheorem{proposition}[theorem]{Proposition}
\newtheorem{corollary}[theorem]{Corollary}
\newtheorem{lemma}[theorem]{Lemma}

\newtheorem{conjecture}[theorem]{Conjecture}

\newtheorem{mainresult}{Theorem}

\newtheorem{mainalgorithm}[mainresult]{Algorithm}

\newtheorem{claim}{Claim}
\crefname{claim}{Claim}{Claims}
\AtEndEnvironment{proof}{\setcounter{claim}{0}}

\theoremstyle{definition}

\newtheorem{example}[theorem]{Example}

\newtheorem*{definition*}{Definition}

\newtheorem{algorithm}[theorem]{Algorithm}
\newtheorem{construction}[theorem]{Construction}
\newtheorem{question}[theorem]{Question}

\theoremstyle{remark}

\usepackage{etoolbox}
\newbool{pdfBool}
\boolfalse{pdfBool} 

\newcommand{\pdfOrNot}[2]{\ifbool{pdfBool}{{#1}}{{#2}}}

\usepackage{svg}

\def\lqedsymbol{\ifmmode$\lrcorner$\else{\unskip\nobreak\hfil
		\penalty50\hskip1em\null\nobreak\hfil$\rule{1.2ex}{1.2ex}$
		\parfillskip=0pt\finalhyphendemerits=0\endgraf}\fi}

\newenvironment{claimproof}[1][\proofname]
{%
	\proof[#1]%
}
{%
	\endproof%
}

 \DeclareMathOperator{\tw}{\mathrm{tw}}

\def\aw{\hbox{\rm aw}_r}
\def\caw{\hbox{\rm caw}_r}
\def\lsw{\hbox{\rm lsw}_r}
\def\ltw{\hbox{\rm ltw}_r}

\DeclareMathOperator{\pre}{pre}

\newcommand{\td}{tree-decom\-pos\-ition}
\newcommand{\gd}{graph-decom\-pos\-ition}

\usepackage{tikz}
\usetikzlibrary{positioning,shapes,calc}

\usepackage{csquotes}
\usepackage[backend=biber, style=alphabetic, maxnames=99, maxalphanames=99, giveninits=true, sortcites=true]{biblatex}
\DeclareLabelalphaTemplate{
    \labelelement{
        \field[final]{shorthand}
        \field{label}
        \field[strwidth=1,strside=left]{labelname}
    }
    \labelelement{
        \field[strwidth=2,strside=right]{year}
    }
}

\begin{document}

\title{The global structure of locally chordal graphs}

\author[T.\ Abrishami \and P.\ Knappe]{Tara Abrishami, Paul Knappe}

\address{Stanford University, Department of Mathematics}
\email{tara.abrishami@stanford.edu}

\address{Universität Hamburg, Department of Mathematics}
\email{paul.knappe@uni-hamburg.de}

\thanks{T.A. was supported by the National Science Foundation Award Number DMS-2303251 and the Alexander von Humboldt Foundation. P.K. was supported by a doctoral scholarship of the Studienstiftung des deutschen Volkes.}

\keywords{Locally chordal, chordal, graph-decomposition, local separator, clique, local covering}

\begin{abstract}
    A graph is \emph{locally chordal} if each of its small-radius balls is chordal. In an earlier work \cite{LocallyChordal}, the authors and Kobler proved that locally chordal graphs can be characterized by having chordal \emph{local covers}, by forbidding short cycles and wheels as induced subgraphs, and by the property that each of their \emph{minimal local separators} is a clique. In this paper, we address the global structure of locally chordal graphs. The global structure of chordal graphs is given by the following characterizations: a graph is chordal if and only if it is the intersection graph of subtrees of a tree, if and only if it admits a tree-decomposition into cliques. We prove a local analog of this characterization, which essentially says that a graph is locally chordal if and only if it is the intersection graph of special subtrees of a high-girth graph, if and only if it admits a special graph-decomposition over a high-girth graph into cliques. We also prove that these global representations of locally chordal graphs can be efficiently computed. 

    This paper has two major contributions. The first is to exhibit for locally chordal graphs an ideal ``local to global'' analysis: given a graph class defined by restricted local structure, we fully describe the global structure of graphs in the class. The second is to develop the theory of \emph{graph-decompositions}. Much of the work in this paper is devoted to properties of graph-decompositions that represent the global structure of graphs. This theory will be useful to find global decompositions for graph classes beyond locally chordal graphs. 
    
\end{abstract}

\maketitle

\section{Introduction}

Consider a graph $G$ which is obtained by gluing cliques in the shape of a high-girth graph $H$ in such a way that, locally at each vertex, the graph $G$ looks as if it were composed of cliques glued together along a tree.
An illustration is given in \cref{fig:chordallocchordalGlobToLoc}.
Graphs that are composed of cliques glued together along a tree are the well-known \emph{chordal}\footnote{Chordal graphs are also referred to as \emph{rigid circuit} \cite{dirac1961rigid} or \emph{triangulated} \cite{Rose} graphs.} graphs.
The global composition of the graph $G$ along $H$ described above ensures that $G$ exhibits such a tree-shape locally: that $G$ is locally chordal.

\begin{figure}[ht]
    \centering
    \includegraphics[width=0.3\linewidth]{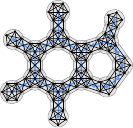}
    \caption{A locally chordal graph $G$, depicted in black, obtained by gluing cliques in the shape of $H$, depicted in blue.}
    \label{fig:chordallocchordalGlobToLoc}
\end{figure}

One way to phrase the above observation is that a global decomposition of $G$ into cliques along a high-girth graph witnesses that $G$ is locally chordal, so knowledge of the global structure allows us to deduce the local structure. One of the great hopes of graph theory research is to be able to do the converse: given knowledge of \emph{local} structure, deduce (something about) the \emph{global} structure. We aim for something even stronger: to \emph{characterize} locally chordal graphs by their global structure. Specifically, we want to show that every locally chordal graph $G$ admits a global decomposition which certifies that $G$ is locally, around each vertex, chordal.

By assumption, a locally chordal graph $G$ looks in the small-radius ball around each vertex like it is composed of cliques glued together along a tree. However, it is not at all obvious that these trees at each vertex can be combined into a single graph $H$ that simultaneously witnesses the local tree-likeness everywhere. Surprisingly, we prove that this is indeed the case: every locally chordal graph $G$ admits a global decomposition that witnesses that $G$ is locally chordal. 

To state our results more precisely, we give some definitions. A graph is \defn{chordal} if each of its cycles of length at least four has a chord. 
For a vertex $v$ of a graph $G$, we define the \defn{ball of radius $r/2$ around $v$}, denoted \defn{$B_{G}(v,r/2)$} or, for short, \defn{$B_{r/2}(v)$}, as the subgraph of $G$ given by the set of all vertices of distance at most $r/2$ from $v$ and the set of all edges of $G$ for which the sum of the endpoints' distances to $v$ is strictly less than $r$.
For an integer $r \geq 0$, the given \emph{degree of locality}, a graph is 
\defn{$r$-locally chordal} if each of its balls of radius $r/2$ is chordal. 
Our main result shows that $r$-locally chordal graphs $G$ admit a global decomposition which simultaneously witnesses that $G$ is $r$-locally chordal:

\begin{restatable}{mainresult}{Characterization}
\label{intro:mainTHM}
    Let $G$ be a locally finite\footnote{Recall that a graph is \defn{locally finite} if each vertex has finite degree. In particular, finite graphs are locally finite.} graph and $r \geq 3$ an integer. The following are equivalent:
    \begin{enumerate}
        \item $G$ is $r$-locally chordal.        \item\label{intro:item:racyclicRIG} $G$ is an $r$-acyclic region intersection graph (\cref{mainresult:racycliccliquegraphs-iff-r-locally-chrodal-etc}).
        \item\label{item:racyclicGD} $G$ admits an $r$-acyclic \gd\ into cliques (\cref{mainresult:racycliccliquegraphs-iff-r-locally-chrodal-etc}).
        \item\label{item:LocAcylicGDs} $G$ admits a \gd\ into cliques which $r$-locally induces \td s (\cref{main:InducingTDsonBalls}).
        \item\label{intro:item:GDrlocSep} $G$ admits a graph-decomposition into cliques which   induces $r$-local separations of $G$ (\cref{mainthm:gd}).
    \end{enumerate}
\end{restatable}

In what follows, we introduce the relevant terms and intuitions.
Our results here draw on our first paper about locally chordal graphs \cite{LocallyChordal}. In that paper, we prove that locally chordal graphs can be characterized by forbidden induced subgraphs, by \emph{minimal local separators} (a local analog of minimal separators), and by their \emph{local covers}. We now state the main theorem of \cite{LocallyChordal}. A graph is \defn{$r$-chordal} if each of its cycles of length at least four and at most $r$ has a chord. A \defn{wheel $W_n$} for $n \geq 4$ consists of a cycle of length $n$, its \emph{rim}, and a vertex $v$ complete to the rim, its \emph{hub}. 

\begin{theorem}[\cite{LocallyChordal}, Theorem 1]\label{BasicCharacterization}
    Let $G$ be a (possibly infinite) graph and let $r \geq 3$ be an integer. The following are equivalent:
    \begin{enumerate}
        \item \label{basic:i} $G$ is $r$-locally chordal.
        \item\label{basic:ii} $G$ is $r$-chordal and wheel-free.
        \item \label{basic:iv} Every minimal $r$-local separator of $G$ is a clique. 
        \item \label{basic:iii} The $r$-local cover $G_r$ of $G$ is chordal.
    \end{enumerate}
\end{theorem}

The equivalence between \cref{BasicCharacterization}~\ref{basic:i} and \ref{basic:iii} is used heavily in this paper. (We also use the equivalence of \ref{basic:i} and \ref{basic:ii}, but we don't use the equivalence to \cref{basic:iv} in this paper.) The local cover is introduced in \cite{canonicalGD} as a graph that represents the local structure of $G$. The local cover is part of a new general theory designed to facilitate local-to-global analysis by constructing global decompositions from local structure. Our results are the first to apply this framework to characterize a locally-defined graph class. We therefore provide a possible roadmap to use this framework for local-to-global analysis more generally.

Next, we describe our results in more detail. 

\subsection{Characterization via region representations}
First, let us discuss \cref{intro:mainTHM}~\cref{intro:item:racyclicRIG}. A \defn{region} of a graph $H$ is a connected subgraph of $H$.
A map $v \mapsto H_v$ assigning to the vertices $v$ of a graph $G$ regions $H_v$ of a graph $H$ \defn{represents} $G$ if an edge $uv$ is present in $G$ if and only if the regions $H_u$ and $H_v$ intersect in at least one vertex.
A region representation $v \mapsto H_v$ is \defn{$r$-acyclic} if for every set $X \subseteq V(G)$ of at most $r$ vertices, the union $\bigcup_{x \in X} H_x$ is an acyclic subgraph of $H$.
If $G$ has such a representation, we then call it an \defn{$r$-acyclic region intersection graph}.

Every graph is a 2-acyclic region intersection graph; see \cref{ex:every-graph-2acyclic} for a proof. But it is already interesting to ask which graphs are 3-acyclic region intersection graphs. 
Wheels, and all graphs containing wheels as induced subgraphs, are not 3-acyclic region intersection graphs (\cref{lem:wheels-not-3acyclic-RIG}), so 3-acyclic region intersection graphs must exclude wheels. In fact, Gavril \cite{nbrhood-chordal} characterized the graphs which are 3-acyclic region intersection graphs as precisely the wheel-free graphs (see \cref{subsec:ComparisonGavril}).
\defn{Wheel-free} graphs, those graphs with no wheel as an induced subgraph, are also precisely the 3-locally chordal graphs. At the other extreme, $\infty$-acyclic region intersection graphs are precisely the intersection graphs of subtrees of forest, which are well known to be exactly the chordal graphs. \cref{intro:mainTHM}~\cref{intro:item:racyclicRIG} now fills in the spectrum for all other values of $r$, giving a characterization of $r$-acyclic region intersection graphs for all $r \in \N$ as exactly the $r$-locally chordal graphs. 

\subsection{Characterization via graph-decompositions}
Let us now turn to  \cref{intro:mainTHM}~\cref{item:racyclicGD}.
A \defn{graph-decomposition} of a graph $G$, as introduced in \cites{canonicalGD}, consists of a model graph~$H$ and an assignment $h \mapsto V_h$ of a vertex-subset $V_h$ of $G$ to each node $h$ of $H$ such that the $V_h$ decompose $G$, i.e.\ (H1) the union of the subgraphs $G[V_h]$ is $G$, and this arranges $G$ in the shape of $H$, i.e.\ (H2) for each vertex $v$ of $G$ the subgraph of $H$ induced by the set $W_v$ of nodes $h$ whose $V_h$ contains $v$ is connected (\cref{subsec:gd-and-rigs}).

A graph-decomposition is \defn{into cliques} if every set $V_h$ is a clique of $G$.
It turns out that \gd s into cliques correspond to region representations. 
More specifically, a \gd\ $(H, \cV)$ of $G$ into cliques induces a region representation of $G$ by mapping each vertex $v$ of $G$ to a connected, spanning subgraph of $H[W_v]$, the induced subgraph from (H2). Conversely, a region representation $v \mapsto H_v$ induces a \gd\ by assigning to each node $h$ of $H$ the vertex-subset $V_h \coloneqq \{v \in V(G) \mid h \in V(H_v)\}$.
Notice that these two constructions are inverse to each other (see \cref{sec:gd-and-rig} for details).

\subsection{Locally inducing tree-decompositions}
Recall that we started with the idea of gluing cliques together in the shape of a high-girth graph so that locally at each vertex the graph looks like being composed of cliques glued together along a tree. 
Graph-decompositions from \cref{intro:mainTHM}~\cref{item:LocAcylicGDs} formalize this idea. Indeed, let $G$ be an $r$-locally chordal graph and $(H, \cV)$ be an $r$-acyclic graph-decomposition of $G$ as in \cref{intro:mainTHM}~\cref{item:LocAcylicGDs}. Now, roughly speaking, if we restrict the decomposition $(H, \cV)$ to just a ball of radius $r/2$ in $G$, we obtain a tree-decomposition into cliques (see \cref{sec:CharViaGDLocDerTD} for more details). Essentially, this means that there is a collection of $|V(G)|$ tree-decompositions into cliques, one for each ball of radius $r/2$ of $G$, such that one can find a single graph-decomposition of $G$ which ``puts them together'' -- in other words, which simultaneously witnesses each local tree-decomposition in a single structure. This property holds for every $r$-acyclic graph-decomposition of locally chordal graphs into cliques (\cref{main:InducingTDsonBalls}).

\subsection{Inducing local separations}
The $r$-acyclic region representations (equivalently: graph-decompositions into cliques), from \cref{intro:mainTHM}~\cref{intro:item:racyclicRIG}, can be characterized as those \gd s which \emph{induce} $r$-local separations (see \cref{sec:GDofLocChord} for details).
\emph{Local separations} were introduced in \cite{computelocalSeps} together with local separators as a local analog to \emph{separations}, which, when $r$ goes to infinity, are precisely the usual separations. Tree-decompositions are well-known to have a special relationship with the usual separations. 
Thus, when $r$ goes to infinity, \cref{intro:mainTHM}~\cref{intro:item:GDrlocSep} gives that connected, chordal graphs are exactly the graph that admit tree-decompositions into cliques.

\subsection{Further results: Computing region representations}

Our proofs of \cref{intro:mainTHM} are not computable; indeed, they rely on the local cover, which is often an infinite graph. Nevertheless, we give an efficient and conceptually simple algorithm (\cref{THEalgorithm}) which, when given an $r$-locally chordal graph $G$, computes an $r$-acyclic region representation of $G$. The algorithm in fact gives an independent proof of the main decomposition theorem \cref{mainresult:racycliccliquegraphs-iff-r-locally-chrodal-etc}, except the existence of a canonical decomposition, for finite locally chordal graphs. Our algorithm is essentially a parallelization of the algorithm given by Gavril \cite{gavrilAlgorithm} to compute 3-acyclic region representations of 3-locally-chordal graphs. 
The algorithm does not depend on $r$. 
The (parallel) running time is at most $O(\Delta^2 \log \Delta)$ on $n = |V(G)|$ processors, where $\Delta$ is the maximum degree of~$G$. Furthermore, the algorithm does not require $r$ as an input (\cref{thm:correctalgo}):
given a wheel-free graph $G$, it returns an $r$-acyclic region representation of $G$ where $r$ is the maximum integer such that $G$ is $r$-locally chordal. 
In the extreme case that the graph $G$ is chordal, the computed region representation is $\infty$-acyclic and thus the graph $H$ is a forest, recovering the familiar notion of {\em intersection graphs of subtrees of a tree} (if $G$ is additionally connected).

More generally,
a graph-decomposition $(H,\cV)$ of a graph $G$ into the maximal cliques of $G$ corresponds to a spanning subgraph of the intersection graph $\bK(G)$ of the maximal cliques in $G$, by identifying the model graph $H$ with the one obtained from the map $h \mapsto V_h$.
Note that, for any fixed vertex $v$ of $G$, the subgraph of $H$ induced on the maximal cliques of $G$ containing $v$ is precisely $H[W_v]$, and thus connected.
Conversely, any given spanning subgraph $H$ of $\bK(G)$ whose subgraphs $H_v$ induced on the maximal cliques containing any fixed vertex $v$ of $G$ are connected yields a region representation of $G$ over $H$ by mapping a vertex $v$ to $H_v$, which then corresponds to a graph-decomposition of $G$ into the maximal cliques of $G$.
We call such an $H$ a \defn{clique graph} of $G$, and it is \defn{$r$-acylic} if its corresponding region representation is.
If $H$ is a tree, this recovers the notion of {\em clique trees}. The algorithm above does indeed return a clique graph of $G$.

Clique trees of a chordal graph $G$ are known to be precisely the maximum-weight spanning trees of $\bK(G)$, where we weight each edge by the size of the corresponding intersection of maximal cliques. We generalize this by proving that $r$-acyclic clique graphs of $r$-locally chordal graphs $G$ are precisely the maximum-weight $r$-acyclic spanning subgraphs of $\bK(G)$ (\cref{thm:maximum-weight-acyclic-spanning}).

Since $r$-acyclic clique graphs can be efficiently computed, they can be used in algorithms.  
Clique trees are powerful algorithmic tools (see e.g.\ \cite{MAL}), and our algorithm to compute $r$-acyclic region representations of locally chordal graphs is an important step toward extending the algorithmic power of chordal graphs to locally chordal graphs. 
We outline a specific potential application, to the maximum independent set problem, in \cref{subsec:maxIndSet}.

\subsection{Acknowledgement}

We thank Nathan Bowler for inspiring input which led to splitting the proof of \cref{mainresult:racycliccliquegraphs-iff-r-locally-chrodal-etc} into \cref{mainthm:wheelfree} and \cref{lem:char-locallychordal-via-3-acyclic-region-rep} and a better overall structure of the paper. We thank Piotr Micek for suggesting \cref{q:k-Helly}. We thank Sandra Albrechtsen and Reinhard Diestel for their feedback on the introduction, which led to its complete overhaul. We thank Raphael W. Jacobs for his input to the paper and especially to the proof of \cref{thm:maximum-weight-acyclic-spanning}. We thank Mai Trinh for useful comments on the paper. We thank Jonas Kobler for interesting discussions and his involvement in the early stages of this work.

\subsection{Notation \texorpdfstring{\&}{and} conventions}

We follow standard graph theory notation and definitions from \cite{bibel}. We also refer to 
\cite{Hatcher} for terminology on topology. 

In this paper we deal with simple graphs, i.e.\ graphs with neither loops nor parallel edges. Often, we will restrict our view to locally finite graphs; a graph is \defn{locally finite} if each vertex has finite degree. We specify when we deal with finite graphs, with locally finite, or with general (possibly infinite) graphs.

Let $G$ be a graph, and let $X$ and $Y$ be two vertex-subsets of $G$. The \defn{induced subgraph} of $G$ on vertex-set $X$ is denoted \defn{$G[X]$}.
The vertex-subset $X$ is \defn{connected} in $G$ if its induced subgraph $G[X]$ is.
An \defn{$X$--$Y$ path} is a path whose first vertex is in $X$, last vertex is in $Y$, and which is otherwise disjoint from $X \cup Y$.
An \defn{$X$-path} is a non-trivial path whose first and last vertex is in $X$ and which is otherwise disjoint from $X$.
A path $P$ is \defn{through $X$} if all internal vertices of $P$ are in $X$ and $P$ has at least one internal vertex. 

Let $d \geq 0$ be a real number.
The \defn{depth-$d$ closed neighborhood $N_G^d[X]$} is the set of all vertices of $G$ that have distance at most $d$ to $X$, and the \defn{depth-$d$ (open) neighborhood $N_G^d(X)$} is the set of all vertices of $G$ that have distance precisely $d$ to $X$. 
If the ambient graph $G$ is clear from context, we also write \defn{$N^d[X]$} and \defn{$N^d(X)$}.
Note that the depth-$d$ open neighborhoods are always empty when $d$ is not an integer.
For any vertex $v$ of $G$, we use \defn{$N_G^d[v]$} and \defn{$N_G^d(v)$} as a shorthand for $N_G^d[\{v\}]$ and $N_G^d(\{v\})$. 
Also we use \defn{$N_G[X]$} as a shorthand for $N^1_G[X]$.

Finally, we will need the following lemma. 

\begin{lemma}[\cite{LocallyChordal}, Lemma 2.4]\label{lem:edge-on-pre-last-level}
    Let $r \geq 4$ be an integer and set $d: = \lfloor r/2 \rfloor$.
    Let $X$ be a connected vertex-subset of a (possibly infinite) chordal graph $G$.
    If, for any two distinct vertices $u,w \in N^{d-1}(X)$, there is a $u$--$w$ path in $B_{r/2}(X)$ through $N^d(X)$, then $uw$ is an edge of $B_{r/2}(X)$.
\end{lemma}

\vspace{9em}
\tableofcontents
\newpage

\section{Equivalence: Graph-decompositions \texorpdfstring{\&}{and} region representations}\label{sec:gd-and-rig}
In this section we introduce our central objects to describe the structure of locally chordal graphs, graph-decompositions and region intersection graphs, and explore the relationship between them.

\subsection{Graph-decompositions \texorpdfstring{\&}{and} region intersection graphs} \label{subsec:gd-and-rigs}

Graph-decompositions, introduced in ~\cite{canonicalGD}, generalize the notion of tree-decompositions by allowing the bags of the decompositions to be arranged along general graphs instead of just trees.

Let $G$ be a graph.
A \defn{\gd} of $G$ is a pair $\cH = (H, \cV)$ consisting of a graph $H$ and a family~$\cV$ of vertex-subsets $V_h$ of $G$ indexed by the nodes $h$ of $H$ which satisfies
\begin{enumerate}[label=(H\arabic*)]
    \item \label{axiomH1} $G = \bigcup_{h \in H} G[V_h]$, and
    \item \label{axiomH2} for every vertex $v$ of $G$, the subgraph $H[W_v]$ of $H$ induced on $W_v := \{h \in V(H) \mid v \in V_h \}$ is connected.
\end{enumerate}
We refer to $H$ as the \defn{model graph}, and also say that $(H, \cV)$ is a \defn{graph-decomposition of $G$ over}, or \defn{modeled on, $H$}.
The vertex-subsets $V_h$ are called the \defn{bags} of $\cH$, and the $W_v$ are called the \defn{co-bags} of $\cH$.
We can flesh out the bags and co-bags by choosing spanning subgraphs $G_h$ of $G[V_h]$ with $G = \bigcup_{h \in H} G_h$ and connected, spanning subgraphs $H_v$ of $H[W_v]$.
Then the $G_h$ are the \defn{parts} and the $H_v$ are the \defn{co-parts} of the decomposition $\cH$.
In this paper, we will always consider the induced subgraphs $G_h := G[V_h]$ of $G$ from \cref{axiomH1} as the \defn{parts}.
Whenever we introduce a \gd\ $(H,\cV)$ without a specific choice of co-parts in this paper, we tacitly assume that we have chosen any valid family of co-parts $H_v$.
A \gd\ is \defn{honest} if every bag is nonempty and the \defn{adhesion sets} $V_{h_1} \cap V_{h_2}$ corresponding to the edges $f = h_1h_2$ of $H$ are nonempty.
Note that a \gd\ is honest if and only if $H = \bigcup_{v \in G} H_{v}$. In this paper, we call a pair $\mathcal{H} = (H, \mathcal{V})$ a \defn{decomposition} if it consists of a graph $H$ and a family $\mathcal{V}$ of vertex-subsets $V_h$ of $G$ indexed by the nodes $h$ of $H$ which satisfies \cref{axiomH1}. 

A \defn{\td} $(T, \cV)$ is precisely a \gd\ $(H,\cV)$ whose model graph $H$ is the tree $T$.
In this case, the only valid choice of co-parts $T_v$ of a \td\ $(T,\cV)$ are the induced subgraphs $T[W_v]$, since the $T[W_v]$ is a tree and co-parts must be connected. Indeed, most of our \gd\ only admit one valid choice of co-parts (cf. \cref{lemma:HvAreInduced}).

The bags $V_h$ and co-bags $W_v$ of any \gd\ $(H,\cV)$ of $G$ satisfy the duality
\begin{equation*}
    v \in V_h \Longleftrightarrow h \in W_v.
\end{equation*}

Given this, we can cryptomorphically axiomatize \emph{\gd s} $(H,\cV)$ (whose parts $G_h$ are the induced subgraphs $G[V_h]$ and whose co-parts $H_v$ are any valid choice), as follows.
A \defn{\gd} of $G$ over a graph $H$ is a map $v \mapsto H_v$ that assigns to each vertex $v$ of $G$ a subgraph $H_v$ of $H$ such that

\begin{enumerate}
    \item[(H1V)]\label{H1V} $H_v$ is nonempty for every vertex $v$ of $G$,
    \item[(H1E)]\phantomsection\label{H1E} $H_u \cap H_v$ is nonempty for every edge $uv$ of $G$, and
    \item[(H2)]\label{H2equiv} $H_v$ is connected for every vertex $v$ of $G$. 
\end{enumerate}

Given a graph $G$, a graph-decomposition $(H, \cV)$ of $G$, and a vertex-subset $X$ of $G$, we set $\mathdefn{H_X} \coloneqq \bigcup_{x \in X} H_x$. Observe that the graph-decomposition axioms and the definition of co-part imply the following: 

\begin{lemma}\label{lem:strongerH2}
    Let $(H,\cV)$ be a graph-decomposition of a graph $G$.
    If a vertex-subset $U$ of $G$ is connected, then~$H_U$ is connected. \qed
\end{lemma}

The \defn{regions} of a graph are its (nonempty) connected subgraphs. 
In this language, a \gd\ over $H$ assigns to each vertex $v$ a region $H_v$ of $H$ such that $H_u \cap H_v \neq \emptyset$ if $uv \in E(G)$. If the map also satisfies the converse of this condition, i.e. $H_u \cap H_v \neq \emptyset$ if and only if $uv \in E(G)$, then we obtain a concept called \emph{region representations}:

A \defn{region representation $v \mapsto H_v$ of $G$ over a graph $H$} is a map that assigns to each vertex $v$ of $G$ a region $H_v$ of $H$ such that $uv$ is an edge of $G$ if and only if $H_u$ and $H_v$ have nonempty intersection. If such a region representation of $G$ over $H$ exists, then $G$ is a \defn{region intersection graph over $H$}.
For a vertex-subset $X$ of $G$, we write \defn{$H_X$} for the union $\bigcup_{x \in X} H_x$ of the regions corresponding to the elements of $X$.
Given a class $\cH$ of graphs, a graph $G$ is a \defn{region intersection graph over $\cH$} if $G$ is a region intersection graph over some $H \in \cH$.

We remark that, by simply restricting the map, a region representation of $G$ induces a region representation of each induced subgraph of $G$.
Then the following is a simple observation similiar to \cref{lem:strongerH2}:

\begin{lemma}\label{lem:strongerH2forRegRep}
    Let $v \mapsto H_v$ be a region representation of a graph $G$.
    Then an induced subgraph $G'$ of $G$ is connected if and only if the union $H_{V(G')}$ of its corresponding regions is connected. \qed
\end{lemma}

A region representation $v \mapsto T_v$ of $G$ over a tree $T$ is better known as a \defn{subtree representation of $G$ over the tree $T$}.
Similarly, if $G$ has a subtree representation over a tree $T$, then $G$ is a \defn{subtree intersection graph over $T$}.

A \defn{clique} $X$ in $G$ is a vertex-subset which induces a complete subgraph of $G$.
It is well-known that, for a \td\ $(T,\cV)$ of a graph $G$ into cliques, the map $v \mapsto T_v$ assigning a vertex to its co-part is a subtree representation and vice versa.
More generally, it is now immediate from the definitions (since the converse of \hyperref[H1E]{(H1E)} holds) that region representation correspond precisely to the \gd s $(H, \cV)$ of $G$ that are decompositions \defn{into cliques}, i.e.\ all its bag $V_h$ are cliques of $G$:

\begin{mainresult}
\label{thm:rig-is-gd-into-cliques}
    Let $G$ and $H$ be (possibly infinite) graphs.
    \begin{enumerate}
        \item  If $v \mapsto H_v$ is a region representation of $G$ over $H$, then $(H, \cV)$ with $V_h \coloneqq \{v \in V(G) \mid h \in V(H_v)\}$ is a graph-decomposition of $G$ into cliques, for which the family $(H_v)_{v \in G}$ is a valid choice of co-parts. 
        \item If $(H, \cV)$ is a graph-decomposition of $G$ into cliques and $(H_v)_{v \in G}$ is a valid choice of co-parts, then $v \mapsto H_v$ is a region representation of $G$ over $H$. \qed
    \end{enumerate}
\end{mainresult}

\noindent We remark that \cref{thm:rig-is-gd-into-cliques}
yields a one-to-one correspondence between region representations of a graph and its \gd s into cliques, which allows us to discuss the \gd\ (into cliques) \defn{corresponding} to a given region representation and vice versa.

\subsection{Canonical graph-decompositions \& region representations}

Let $G$ be a graph and let $\Gamma$ be a subgroup of the group \defn{$\Aut(G)$} of automorphisms of $G$.
A \gd\ $(H, \cV)$ of $G$ is \defn{$\Gamma$-canonical} if there exists an action of $\Gamma$ on $H$ which commutes with the assignment $h \mapsto V_h$ of the bags, i.e.\ $\gamma (V_h) = V_{\gamma \cdot h}$ for every $\gamma \in \Gamma$.
Equivalently, the action of $\Gamma$ on $H$ commutes with the assignment $v \mapsto W_v$ of the co-bags, i.e.\ $W_{\gamma(v)} = \gamma \cdot W_v$ for every $\gamma \in \Gamma$.
If parts $G_h$ or co-parts $H_v$ are specified, then we require the respective assignments $h \mapsto G_h$ or $v \mapsto H_v$ to commute with the action of $\Gamma$ on $H$.
Moreover, $(H, \cV)$ is \defn{canonical} if it is $\Aut(G)$-canonical.

As discussed in \cref{subsec:gd-and-rigs}, we can consider a \gd\ $(H,\cV)$ with a family $(H_v)_{v \in G}$ of specified co-parts to be presented by the assignment $v \mapsto H_v$ of co-parts. 
Then the \gd\ $(H,\cV)$ is $\Gamma$-canonical if and only if $v \mapsto H_v$ commutes with the action of $\Gamma$ on $H$, i.e.\ $\gamma \cdot H_v = H_{\gamma(v)}$.
Thus, we say that a region representation $v \mapsto H_v$ of $G$ over a graph $H$ is \defn{$\Gamma$-canonical} or \defn{canonical} if its corresponding \gd\ with specified co-parts $H_v$ is $\Gamma$-canonical or canonical, respectively.

\subsection{Clique graphs \texorpdfstring{\&}{and} graph-decompositions into maximal cliques}\label{subsec:cliquegraphs}

Let $G$ be a graph.
We denote the set of all (inclusion-wise) maximal cliques of $G$ by \defn{$K(G)$}.
By \defn{$\bK(G)$} we denote the intersection graph of the maximal cliques of $G$, i.e.\ the graph $\bK(G)$ whose vertex set is $K(G)$ and where two maximal cliques $K_1$ and $K_2$ of $G$ are adjacent in $\bK(G)$ when $K_1 \cap K_2 \neq \emptyset$.
The set of maximal cliques of $G$ containing a fixed vertex $v$ of $G$ we denote by \defn{$K_G(v)$}.

The map $v \mapsto \bK(G)[K_G(v)]$ always forms a \gd\ of $G$ over $\bK(G)$.
A naive attempt to flesh out the $\bK(G)[K_G(v)]$ by a choice of co-parts $H_v$, i.e.\ connected, spanning subgraphs, is to choose a spanning subgraph $H$ of $\bK(G)$ and consider the induced subgraph $H[K_G(v)]$ as co-parts $H_v$.
This then yields a valid choice of co-parts for $v \mapsto \bK(G)[K_G(v)]$ if and only if the $H[K_G(v)]$ are all connected.

More formally, given any spanning subgraph $H$ of $\bK(G)$, the assignment $v \mapsto H_v \coloneqq H[K_G(v)]$ will always satisfy $H_u \cap H_v \neq \emptyset$ if and only if $uv \in E(G)$ (and also \hyperref[H1V]{(H1V)} and \hyperref[H1E]{(H1E)} are satisfied).
But, $v \mapsto H_v$ might not be a region representation (equivalently: \gd) of $G$, as these $H_v$ might not be connected.
Those that do induce a region representation we call \emph{clique graphs} of $G$:
a \defn{clique graph}\footnote{In the literature the graph $\bK(G)$ is sometimes also called \emph{clique intersection graph} or \emph{clique graph}. As our \emph{clique graphs} are the object analogous to \emph{clique trees}, we divert from this standard and refer to $\bK(G)$ instead as the \emph{intersection graph of the maximal cliques of $G$}.} of $G$ is a spanning subgraph $H$ of $\bK(G)$ such that $v \mapsto H_v \coloneqq H[K_G(v)]$ is a region representation of $G$ over $H$ (equivalently: $H[K_G(v)]$ is connected for every vertex $v$ of $G$).

Conversely, let us consider a \gd\ $(H,\cV)$ of $G$ \defn{into the maximal cliques of $G$}, i.e.\ the map $h \mapsto V_h$ is a bijection from the nodes of $H$ to the set $K(G)$ of all maximal cliques of $G$.
If the $H_v$ are induced subgraphs of $H$ and the decomposition is honest, then the map $h \mapsto V_h$ is an isomorphism from $H$ to a spanning subgraph $H'$ of $\bK(G)$ such that the $H_v$ map precisely on $H'[K_G(v)]$.
In such cases, we sometimes just refer to the spanning subgraph $H'$ of $\bK(G)$ to describe the \gd\ $(H,\cV)$ into maximal cliques.

The above is a one-to-one correspondence between clique graphs of a graph and its (honest) \gd s into maximal cliques (up to the images of the model graphs in $\bK(G)$), which allows us to discuss the \gd\ (into its maximal cliques) \defn{corresponding} to a given clique graph and vice versa:

\begin{lemma}\label{lem:cg-is-gd-into-max-cliques}
    Let $G$ be a (possibly infinite) graph.
    \begin{enumerate}
        \item If $H$ is a clique graph of $G$, then assigning to each node $K$ of $H$ (where $K \in K(G)$) the bag $V_K \coloneqq K$ yields a graph-decomposition $(H,\cV)$ of $G$ into the maximal cliques of $G$.
        \item If $(H,\cV)$ is a \gd\ of $G$ into the maximal cliques of $G$, then the map $h \mapsto V_h$ embeds $H$ onto a clique graph of $G$. \qed
    \end{enumerate}
\end{lemma}

We recall that a \defn{clique tree}\footnote{Clique trees are also known as \emph{junction trees}, see e.g.\ \cite{MAL}.} of $G$ is a clique graph which is a tree. 
Equivalently said, a clique tree of $G$ is a spanning tree $T$ of $\bK(G)$ such that $v \mapsto T_v \coloneqq T[K_G(v)]$ is a subtree representation of $G$ over $T$ (equivalently: $T[K_G(v)]$ is connected for every vertex $v \in G$).

The following properties of spanning subgraphs $H$ of $\bK(G)$ and the maps $v \mapsto H[K_G(v)]$ they induce immediately follow from the definitions:

\begin{lemma}\label{propertiesofspanningsubgraph-of-CIG}
    Let $G$ be a (possibly infinite) graph. Suppose $H$ is a spanning subgraph of $\bK(G)$, and set $H_v \coloneqq H[K_G(v)]$ for every vertex $v$ of $G$.
    Then $K_G(v) = K(B_{3/2}(v))$ for every vertex $v$ of $G$.
    Moreover, for every two vertices $u,v$ of $G$, we have
    \begin{enumerate}
        \item $K_{B_{3/2}(v)}(u) = K_G(u) \cap K_G(v)$, and
        \item $H_u \cap H_v = H[K_G(u) \cap K_G(v)] = H_v[K_{B_{3/2}(v)}(u)]$. \qed
    \end{enumerate}
\end{lemma}

\noindent 

For the curious reader, we remark that one may generalize $K_G(v) = K(B_{3/2}(v))$ to $\bigcup_{v \in N^{r/2-1}[v]} K_G(v) = K(B_{r/2}(v))$ for $r$-locally chordal graphs with odd $r \geq 3$.

\subsection{Canonical clique graphs}
A clique graph $H$ of $G$ is \defn{$\Gamma$-canonical} for a subgroup $\Gamma$ of $\Aut(G)$ or \defn{canonical} if its corresponding region representation $v \mapsto H[K_G(v)]$ of $G$ is $\Gamma$-canonical or canonical, respectively.
The automorphisms $\gamma$ of $G$ induce an action on the set of all maximal cliques of $G$, and thus on $\bK(G)$, by $K \mapsto \gamma(K)$.
Note that a clique graph $H$ is $\Gamma$-canonical if and only if $H$ is $\Gamma$-invariant, i.e.\ $\gamma(H) = H$ for every $\gamma \in \Gamma$.

\section{Overview: Region representations of (locally) chordal graphs}

In this section, we discuss why region representations  (equivalently: graph-decompositions into cliques) are the right tools to describe the structure of locally chordal graphs.
First, we review the classic characterizations of chordal graphs as those graphs which have subtree representations (equivalently: admit tree-decompositions into cliques). 
Then, we explain how these characterizations generalize to the locally chordal case.

It is a simple observation\footnote{This statement is immediate from the first paragraph of the proof presented in \cite[Theorem~3]{GavrilChordalGraphsSTIG}.} that subtree intersection graphs are chordal:

\begin{proposition}[Folklore]\label{prop:subtreeintersec-is-chordal}
    Every (possibly infinite) subtree intersection graph is chordal
    (equivalently: every graph that admits a \td\ into cliques is chordal).
\end{proposition}

Conversely, Gavril in the finite case \cite{GavrilChordalGraphsSTIG} and Halin \cite[10.2']{halin1964simpliziale} based on his \emph{zweiter Zerlegungssatz} \cite[Satz 5]{halin1989graphentheorie} showed that in particular every locally finite, connected, chordal graph is a subtree intersection graph (equivalently: admits a \td\ into cliques).
There are two stronger variants of subtree representations (equivalently: \td\ into cliques): clique trees and canonical ones.
Indeed, Halin showed in \cite{halin1964simpliziale} that every locally finite, connected, chordal graph admits a clique tree.\footnote{Halin actually showed this for connected, chordal graphs that have no infinite maximal clique. Later, Hofer-Temmel and Lehner~\cite[Theorem~3]{chordalinfinite} extended the techniques from finite graphs to locally finite graphs to obtain another proof of this result in the case of locally finite graphs.}
Jacobs and Knappe \cite[Theorem 2]{canTDofChordalGraphs} showed that every such graph also admits a canonical \td\ into cliques.

The following characterization of chordal graphs is immediate by applying the above-mentioned results to the components of the considered graph.
A region representation $v \mapsto H_v$ of a graph $G$ over a graph $H$ is \defn{acyclic} if $H_{V(G)}$ is acyclic.
A clique graph $H$ of $G$ is \defn{acyclic} if its corresponding region representation is acyclic (equivalently: $H$ is acyclic).
A \defn{forest-decomposition} of a graph $G$ is a graph-decomposition $(H, \cV)$ where $H$ is a forest. 

\begin{theorem}\label{thm:chordal}
    Let $G$ be a locally finite (not necessarily connected) graph.
    Then the following are equivalent:
    \begin{enumerate}
        \item $G$ is chordal.
        \item $G$ has an acyclic clique graph (equivalently: $G$ admits a forest-decomposition into maximal cliques).
        \item\label{item:chordal:canonical} $G$ has a canonical acyclic region representation (equivalently: $G$ admits a canonical forest-decomposition into cliques).
        \item $G$ has an acyclic region representation (equivalently: $G$ admits a forest-decomposition into cliques).
    \end{enumerate}
\end{theorem}

In contrast, every graph $G$ is a region intersection graph, since the intersection graph $\bK(G)$ of maximal cliques itself is already a clique graph of $G$.
Thus, we need to require additional properties to use region representations to characterize locally chordal graphs.

As cycles are finite, a region representation $v \mapsto H_v$ of $G$ over a graph $H$ is acyclic if and only if only if for every finite vertex-subset $X$ of $G$ the union $H_X$ is acyclic.
More generally, let $H$ be a graph and let $\bH$ be a family of subgraphs of $H$. 
Then, $\bH$ is \defn{$r$-acyclic} for an integer $r \geq 0$ if every union of at most $r$ elements in $\bH$ is acyclic. 
A region representation $v \mapsto H_v$ of a graph $G$ over $H$ is \defn{$r$-acyclic} if the family $(H_v \mid v \in V(G))$ is $r$-acyclic.
A clique graph of $G$ is \defn{$r$-acyclic} 
if its 
corresponding region representation is $r$-acyclic.
If a graph $G$ has an $r$-acyclic region representation, then $G$ is an \defn{$r$-acyclic region intersection graph}.
A \gd\ $(H,\cV)$ with a fixed family $(H_v)_{v \in G}$ of co-parts is \defn{$r$-acyclic} if the family $(H_v)_{v \in G}$ is $r$-acyclic.

Let us observe that honest $r$-acyclic graph-decompositions always have high-girth models, i.e.\ the models are locally trees.

\begin{lemma}\label{lem:H-high-girth}
Let $r \geq 3$ be an integer, let $G$ be a (possibly infinite) graph, and let $(H, \cV)$ be an honest, $r$-acyclic graph-decomposition of $G$. Then $H$ has girth greater than $r$. 
\end{lemma}

\begin{proof}
Let $C$ be a cycle of $H$ of length $\ell \geq 3$. Since $(H,\cV)$ is honest, we may choose for each edge $f$ of $C$ a vertex $v_f$ of $G$ in its adhesion set $V_f$. Now, $C \subseteq \bigcup_{f \in E(C)} H_{v_f}$. Since $(H, \cV)$ is $r$-acyclic, it follows that $\ell > r$. Therefore, $H$ has no cycle of length at most $r$, i.e.\ $H$ has girth greater than $r$. 
\end{proof}

In light of the notion of \emph{$r$-acyclic}, we may refer to acyclic region-representations, families of subgraphs, or clique graphs as being \defn{$\infty$-acyclic}.
Then \cref{thm:chordal} says that chordal graphs are exactly the graphs which have $\infty$-acyclic region representations. On the other end of the spectrum, every graph has not only a region representation but indeed even a $(g(G)-1)$-acyclic region representation, where $\mathdefn{g(G)} \geq 3$ is the \defn{girth} of $G$, i.e.\ the length of a shortest cycle in $G$.

\begin{example}\label{ex:every-graph-2acyclic}
    Every graph $G$ has a $(g(G)-1)$-acyclic region representation $v \mapsto H_v$ over some graph $H$ such that the bag $V_h = \{v \in V(G) \mid h \in V(H_v)\}$ is a clique of size at most $2$ for every node $h$ of $H$.
\end{example}

\begin{proof}
    Let $G$ be a graph and let $H$ be the graph formed from $G$ by subdividing each edge $e$ once by a vertex $v_e$.
    Then assign each vertex $v$ of $G$ the star $S_v$ in $H$ formed by $v$ together with $v_e$ for all edges $e$ incident to $v$ in $G$. 
    Note that the stars $S_v$ are edge-disjoint, and every cycle in the subdivision $H$ of $G$ meets every $S_v$ in precisely $0$ or $2$ edges.
    Thus, $v \mapsto H_v$ is a $(g(G)-1)$-acyclic region representation of $G$, the bag $V_{v_e}$ is $\{u,v\}$ for every edge $e = uv$ of $G$, and the bag $V_v$ is $\{ v\}$ for every vertex $v$ of $G$.
\end{proof}

Given \cref{thm:chordal} and \cref{ex:every-graph-2acyclic}, it is natural to ask whether there is a nice characterization $r$-acyclic region intersection graphs for general values of $r$. The class of $r$-acyclic region intersection graphs is (inclusion-wise) monotone decreasing as $r$ goes to infinity, eventually approaching the class of chordal graphs. 
It turns out that $r$-locally chordal graphs are exactly those graphs. 

\begin{mainresult}\label{mainresult:racycliccliquegraphs-iff-r-locally-chrodal-etc}
    Let $G$ be a locally finite graph and $r \geq 3$ an integer.
    Then the following are equivalent:
    \begin{enumerate}
        \item $G$ is $r$-locally chordal.
        \item $G$ has an $r$-acyclic clique graph (equivalently: $G$ admits an $r$-acyclic \gd\ into maximal cliques).
        \item $G$ has an $r$-acyclic canonical region representation (equivalently: $G$ admits a canonical $r$-acyclic \gd\ into cliques).
        \item $G$ has an $r$-acyclic region representation (equivalently: $G$ admits an $r$-acyclic \gd\ into cliques).
    \end{enumerate}
\end{mainresult}

\subsection{Proof of \texorpdfstring{\cref{mainresult:racycliccliquegraphs-iff-r-locally-chrodal-etc}}{Theorem XX}}

The proof of \cref{mainresult:racycliccliquegraphs-iff-r-locally-chrodal-etc} involves two steps. First, we prove that \cref{mainresult:racycliccliquegraphs-iff-r-locally-chrodal-etc} holds for $r = 3$: 
\begin{restatable}{mainresult}{acyclicRegionRepOfWheelFree}\label{mainthm:wheelfree}
    Let $G$ be a locally finite graph. The following are equivalent: 
    \begin{enumerate}
        \item \label{wheelfree:i} $G$ is $3$-locally chordal (equivalently:  wheel-free). 
        \item \label{wheelfree:ii} $G$ has a 3-acyclic clique graph. 
        \item \label{wheelfree:iii} $G$ has a 3-acyclic canonical region representation. 
        \item \label{wheelfree:iv} $G$ has a 3-acyclic region representation.  
    \end{enumerate}
\end{restatable}

Next, we prove that for $r > 3$, the $3$-acyclic region representations of $r$-locally chordal graphs are already the $r$-acyclic region representations: 

\begin{restatable}{mainresult}{CharViaAcyclicRegionRep}\label{lem:char-locallychordal-via-3-acyclic-region-rep}
    Let $v \mapsto H_v$ be a $3$-acyclic region representation of a (possibly infinite) graph $G$ over a graph $H$, and let $r > 3$ be an integer.
    Then $G$ is $r$-locally chordal if and only if $v \mapsto H_v$ is $r$-acyclic.
\end{restatable}

The proof of \cref{mainthm:wheelfree} appears in \cref{sec:wheel-free} and the proof of \cref{lem:char-locallychordal-via-3-acyclic-region-rep} appears in \cref{sec:3acyclic-racyclic}. Before proceeding to the proofs, we remark that 3-locally-chordal graphs have been analyzed through a similar view in the past. Specifically, in the nineties, Gavril \cite{nbrhood-chordal} studied (finite) 3-locally-chordal graphs, which he called {\em neighborhood chordal}, and obtained a characterization similar to that in \cref{mainthm:wheelfree}. Next, we discuss his result and how it relates to ours. 

\subsection{Comparison to Gavril's work}\label{subsec:ComparisonGavril}

Gavril \cite[Theorem~5]{nbrhood-chordal} characterized the (finite) graphs that are $3$-locally chordal (equivalently: wheel-free) as precisely the graphs which have Helly, $2$-acyclic region representations.
To state his result, let us first introduce the property \emph{Helly}.
Let $H$ be a graph and let $\bH$ be a family of subgraphs of $H$. Then, $\bH$ is \defn{(finitely) Helly}\footnote{If we restrict this notion to finite graphs $H$, this is the same as Gavril's definition of \emph{Helly}.} if every finite set of pairwise intersecting elements of $\bH$ has nonempty intersection.
A region representation $v \mapsto H_v$ of a graph $G$ over $H$ is \defn{Helly}
if $(H_v \mid v \in V(G))$ is Helly.
A clique graph of $G$ is \defn{Helly}
if its corresponding region representation is.
A graph-decomposition $(H,\cV)$ is \defn{Helly} if the family $(H_v \mid v \in V(G))$ of its co-parts is Helly.

\begin{theorem}[Gavril \cite{nbrhood-chordal}, Theorems 4 \& 5]\label{thm:GavrilnbhdchordalChar}
    Let $G$ be a finite graph. The following are equivalent:
    \begin{enumerate}
        \item $G$ is $3$-locally chordal (equivalently: wheel-free).
        \item $G$ has a Helly, $2$-acyclic clique graph.
        \item $G$ has a Helly, $2$-acylic region representation.
    \end{enumerate}
\end{theorem}

We show in this section that \cref{thm:GavrilnbhdchordalChar} is the same as \cref{mainresult:racycliccliquegraphs-iff-r-locally-chrodal-etc}~(without (iii)) for $r = 3$ (and finite graphs $G$). 
For this, it suffices to prove the following \cref{lem:Helly2-acyclic-equiv-3acyclic}: a region representation is $3$-acyclic if and only if it is Helly and $2$-acyclic.
We remark that this then also yields that the region representations in \cref{mainresult:racycliccliquegraphs-iff-r-locally-chrodal-etc} are Helly.

\begin{lemma}\label{lem:Helly2-acyclic-equiv-3acyclic}
    Let $\bH$ be a family of regions of a (possibly infinite) graph $H$. Then $\bH$ is $3$-acylic if and only if $\bH$ is $2$-acyclic and Helly.
\end{lemma}

To prove \cref{lem:Helly2-acyclic-equiv-3acyclic} not only for finite but also infinite graphs $H$, we first extend two of Gavril's statements on finite graphs $H$ to infinite graphs $H$.

\begin{lemma}[Gavril \cite{nbrhood-chordal}, Lemma~1 and Theorem~2]
\label{lem:gavril-lemmas}
    Let $\bH$ be a 2-acyclic family of regions of a (possibly infinite) graph $H$.
    \begin{enumerate}
        \item \label{gavril:1} If the elements in $\bH$ have a node in common, then the union $\bigcup \bH$ of all elements in $\bH$ is a tree. 
        \item \label{gavril:2} $\bH$ is Helly if and only if the intersection of every three pairwise intersecting elements of $\bH$ is nonempty.
    \end{enumerate}
\end{lemma}

\begin{proof}
    Gavril \cite[Lemma~1 and Theorem~2]{nbrhood-chordal} proved \cref{gavril:1} and \cref{gavril:2} for finite graphs $H$, respectively. 
    In the following, we reduce the case that $H$ is infinite to the case that $H$ is finite.

    \begin{claim}
        \cref{gavril:1} holds for infinite $H$, as well.
    \end{claim}
    
    \begin{claimproof}
    The union $\bigcup \bH$ is connected, as the elements of $\bH$ are connected and all share a node, say $t$.
    Suppose for a contradiction that $\bigcup \bH$ contains a cycle $C$. 
    Then there is a finite subfamily $\bH^* = (H_i \mid i \in [n])$ of $\bH$ such that $C \subseteq \bigcup \bH^*$.
    For each $H_i$, we may choose a finite connected subgraph $H'_i$ of the connected graph $H_i$ that contains $t$ and the finite subgraph $C \cap H_i$ of $H_i$.
    As $H_i' \subseteq H_i$ and $\bH$ is $2$-acyclic, the family $\bH'$ consisting of the regions $H_i'$ of $H'$ is $2$-acyclic.
    Their union $H' \coloneqq \bigcup \bH'$ is a finite graph.
    The choice of the $H_i'$ ensures that $t$ is a common node of the $H_i'$.
    Applying Gavril's \cref{gavril:1} for finite graphs to the finite family $\bH'$ of finite regions in the finite graph $H'$ yields that $\bigcup \bH'$ is a tree.
    This contradicts that, by the choice of the $H_i'$, their union $\bigcup \bH'$ contains the cycle $C$.
    \end{claimproof}

    \begin{claim}
        \cref{gavril:2} holds for infinite $H$, as well.
    \end{claim}
    
    \begin{claimproof}
    Let $\bH^* = (H_i \mid i \in [n])$ be a finite subfamily of pairwise intersecting elements of $\bH$. 
    For each two $H_i,H_j$ in $\bH^*$, fix a node $h_{i,j} = h_{j,i}$ in their intersection $H_i \cap H_j$. 
    Now, for each $H_i$ in $\bH^*$, we may choose a finite connected subgraph $H_i'$ of the connected graph $H$ that contains $h_{i,j}$ for every $H_j$ in $\bH^*$, since $\bH^*$ is finite and $H_i$ is connected.
    As $H'_i \subseteq H_i$ and $\bH$ is $2$-acyclic, the family $\bH' \coloneqq (H_i' \mid i \in [n])$ is $2$-acyclic.
    Their union $H'\coloneqq \bigcup \bH'$ is a finite graph.
    The choice of the $H_i'$ ensures that each two elements $H_i', H_j'$ of $\bH'$ still intersect in $h_{i,j} = h_{j,i}$.
    Now, applying Gavril's \cref{gavril:2} for finite graphs to the finite family of finite regions in the finite graph $H'$ yields that the intersection $\bigcap \bH'$ of all elements of $\bH'$ is nonempty.
    Since $\bigcap \bH' \subseteq \bigcap \bH^*$, also $\bigcap \bH^*$ is nonempty.
    \end{claimproof}
    This completes the proof.
\end{proof}

Additionally, we need the following lemma to deduce that a Helly and $2$-acyclic family is $3$-acyclic.

\begin{lemma}\label{lem:2acyclic-yields-pw-intersecting}
    Let $\bH$ be a $2$-acyclic family of regions of a (possibly infinite) graph $H$.
    If the union of three elements of $\bH$ contains a cycle, then each two of these regions intersect.
\end{lemma}

\begin{proof}
    Let $H_1,H_2,H_3$ be three regions in $\bH$ whose union contains a cycle $C$.
    Note that the $H_i$ are trees, as $\bH$ is $1$-acyclic.
    Moreover, since $\bH$ is $2$-acyclic, the cycle $C$ contains an edge $f_i$ of $H_i$ which is not contained in any other $H_j$ with $j \neq i$.
    Consider the path $P_i$ in $C$ which contains $f_i$ and such that the intersection of $P_i$ with $\bigcup_{j \neq i} H_j$ is exactly the ends of $P_i$; in particular, $P_i \subseteq H_i$.
    Since $\bH$ is $2$-acyclic and $P_i$ contains the edge $f_i$ that is not in $H_j$ with $j \neq i$, the ends of $P_i$ are not both in the same $H_j$ for a $j \neq i$.
    The ends of the $P_i$ thus witness that the desired intersections are nonempty.
\end{proof}

\begin{corollary}\label{lem:2acyclic-yields-Xcontainscycle-size3-is-clique}
    Let $v \mapsto H_v$ be a region representation of a graph $G$ over a (possibly infinite) graph $H$.
    Assume that $X$ is a vertex-subset of $G$ of minimum size such that $H_X$ contains a cycle.
    If $X$ has size at most $3$, then $X$ is a clique of $G$.
    \qed
\end{corollary}

\begin{proof}
    If $X$ has size $1$ or $2$, then the statement is immediate.
    If $X$ has size $3$, then the family $(H_v \mid v \in V(G))$ is $2$-acyclic, and thus the statement follows from \cref{lem:2acyclic-yields-pw-intersecting}.
\end{proof}

The following well-known fact about trees will be frequently used in the rest of the paper. We include a proof for completeness. 

\begin{lemma}\label{lem:glueingtrees-to-trees}
    Let $T$ and $T'$ be two trees.
    Then the following three statements are equivalent:
    \begin{enumerate}
        \item\label{item:glueingtrees:i} $T \cup T'$ is a tree.
        \item\label{item:glueingtrees:ii} $T \cap T'$ is connected.
        \item\label{item:glueingtrees:iii} $T \cap T'$ is a tree.
    \end{enumerate}
\end{lemma}

\begin{proof}
    We prove \cref{item:glueingtrees:i}$\to$\cref{item:glueingtrees:ii}$\to$\cref{item:glueingtrees:iii}$\to$\cref{item:glueingtrees:i}. 

    \begin{claim}
        \cref{item:glueingtrees:i} implies \cref{item:glueingtrees:ii}.
    \end{claim}

    \begin{claimproof}
        Assume \cref{item:glueingtrees:i}, i.e.\ $T \cup T'$ is a tree.
        For every two given vertices $s,t$ of $T \cap T'$, their unique $s$--$t$ path $P$ in $T \cup T'$ is the unique $s$--$t$ path in $T$ and also in $T'$.
        Hence, $P$ is an $s$--$t$ path in $T \cap T'$. 
        Thus, $T \cap T'$ is connected, i.e.\ \cref{item:glueingtrees:ii} holds.
    \end{claimproof}

    \begin{claim}
        \cref{item:glueingtrees:ii} implies \cref{item:glueingtrees:iii}.
    \end{claim}

    \begin{claimproof}
        Any connected subgraph of a tree is a tree.
        Thus, the, by \cref{item:glueingtrees:ii}, connected subgraph $T \cap T'$ of the tree $T$ is a tree, as desired.
    \end{claimproof}

    \begin{claim}
        \cref{item:glueingtrees:iii} implies \cref{item:glueingtrees:i}.
    \end{claim}

    \begin{claimproof}
        Assume \cref{item:glueingtrees:iii}, i.e.\ $T \cap T'$ is a tree. 
        In particular, $T \cap T'$ is nonempty.
        Thus, $T \cup T'$ is connected, as $T$ and $T'$ are trees.
        To show \cref{item:glueingtrees:i}, it remains to show that $T \cup T'$ is acyclic.
        Suppose for a contradiction that $T \cup T'$ is not acyclic.
        As $T, T'$ are trees, we may pick cycle $C$ in $T \cup T'$ such that $P\coloneqq T \cap C$ and $Q \coloneqq T' \cap C$ are paths.
        But as the endvertices of $P$ (equivalently: $Q$) are in $T \cap T'$, and thus there is a path $R$ in the tree $T \cap T'$ joining them.
        Now $R$ is either distinct from $P$ or from $Q$.
        Hence $R \cup P \subseteq T$ or $R \cup Q \subseteq T'$ contains a cycle, which is a contradiction to $T,T'$ being trees.
    \end{claimproof}

    This completes the proof.    
\end{proof}

We are now ready to prove \cref{lem:Helly2-acyclic-equiv-3acyclic}. 

\begin{proof}[Proof of \cref{lem:Helly2-acyclic-equiv-3acyclic}]
    Let $\bH$ be a family of regions of a graph $H$.

    \begin{claim}
        If $\bH$ is Helly and $2$-acyclic, then $\bH$ is $3$-acyclic.
    \end{claim}

    \begin{claimproof}
        Consider a three-element subset $\bH'$ of $\bH$. 
        Suppose for a contradiction that their union contains a cycle.
        Then \cref{lem:2acyclic-yields-pw-intersecting} yields that they are pairwise intersecting.
        Thus, since $\bH$ is Helly, the elements of $\bH'$ have a node in common. By \cref{lem:gavril-lemmas}~\labelcref{gavril:1}, it follows that the union of $\bH'$ is a tree, and thus acyclic. 
    \end{claimproof}

    \begin{claim}
        If $\bH$ is $3$-acyclic, then $\bH$ is Helly and $2$-acyclic.
    \end{claim}

    \begin{claimproof}
        Since $\bH$ is $3$-acyclic, $\bH$ is also $2$-acyclic. 
        It remains to show that $\bH$ is Helly.
        By \cref{lem:gavril-lemmas}~(ii), it suffices to consider three pairwise intersecting elements $H_1,H_2,H_3$ of $\bH$ and show that their intersection is nonempty.
        Since $\bH$ is $2$-acyclic, $H_1 \cup H_2$ is acyclic. 
        So, since their intersection $S \coloneqq H_1 \cap H_2$ is nonempty but $H_1$ and $H_2$ are connected, $H_1 \cup H_2$ is connected, and thus a tree.
        If $H_3$ meets the intersection $S = H_1 \cap H_2$, then we are done. 
        Thus, we assume that $H_3$ does not intersect $S$.
        Fix $x_i \in H_i \cap H_3$ for $i = 1,2$ witnessing that $H_i$ and $H_3$ intersect.
        In particular, $x_i \notin S$ by assumption, so $x_i \not \in H_{3-i}$. 
        Let $P$ be the path from $x_1$ to $x_2$ in $H_3$, and let $Q$ be the path from $x_1$ to $x_2$ in $H_1 \cup H_2$. 
        In particular, $P$ and $Q$ are distinct because $P \subseteq H_3$ avoids $S$ and, since $x_i  \in H_i \setminus H_{3-i}$ for $i = 1,2$, the path $Q \subseteq H_1 \cup H_2$ meets $H_1 \cap H_2 = S$. 
        Hence, $P \cup Q \subseteq \bigcup_{i=1}^3 H_i$ contains a cycle, contradicting that $\bH$ is 3-acyclic. Therefore, $\bH$ is Helly.
    \end{claimproof}
    Therefore, 3-acyclic is equivalent to 2-acyclic and Helly for region representations. 
\end{proof}

\section{Region representations of wheel-free graphs}\label{sec:wheel-free}

In this section, we prove \cref{mainthm:wheelfree}, which we restate here.

\acyclicRegionRepOfWheelFree*

In view of \cref{lem:Helly2-acyclic-equiv-3acyclic}, Gavril~\cite[Theorems~4 and~5]{nbrhood-chordal} showed the equivalences of all the above but \cref{wheelfree:iii} for finite graphs $G$.
We revisit Gavril's construction in \cref{sec:ComputingCliqueGraphs} in an algorithmic context. 
In this section, our proof of \cref{mainthm:wheelfree} follows a new approach. First, we give a basic proof that \cref{wheelfree:iv} implies \cref{wheelfree:i} in \cref{subsec:properties-of-3acyclic}. But the main effort of this section, beginning in \cref{subsec:background-local-coverings}, is in proving that \cref{wheelfree:i} implies \cref{wheelfree:ii} and that \cref{wheelfree:i} implies \cref{wheelfree:iii}. To do this, we draw on and expand the recent theory of local structure from \cites{canonicalGD, computelocalSeps}. Essentially, we start from the fact, proven in \cite{LocallyChordal}, that $G$ is $r$-locally chordal if and only if its $r$-local cover $G_r$ is chordal (\cref{BasicCharacterization}). Chordal graphs admit clique trees and tree-decompositions into cliques, so the local cover $G_r$ admits these decompositions. Finally, the machinery of local theory can be used to turn decompositions of $G_r$ into the decompositions of $G$ that we seek.

\subsection{Properties of 3-acyclic region representations}\label{subsec:properties-of-3acyclic}
We prove two lemmas about 3-acyclic region representations. 
\begin{lemma}\label{lem:bag-cont-clique-in-3-acyclic}
    Let $(H,\cV)$ be a $3$-acyclic graph-decomposition of a connected (possibly infinite) graph $G$.
    If $K$ is a finite clique of $G$, then some bag $V_h$ of $(H,\cV)$ contains $K$.

    In particular,
    let $v \mapsto H_v$ be a $3$-acyclic region representation of a connected (possibly infinite) graph $G$ over some graph $H$.
    If $K$ is a finite clique of $G$, then $\bigcap_{v \in K} H_v$ is nonempty.
\end{lemma}

\begin{proof}
    By \cref{lem:Helly2-acyclic-equiv-3acyclic}, $(H,\cV)$ is Helly.
    Since $K$ is a clique, \cref{axiomH1} yields that the co-parts $H_v$ with $v \in K$ are pairwise intersecting.
    Therefore, by Helly, there is some node $h$ in the intersection of all the co-parts $H_v$ with $v \in K$, i.e.\ $V_h \supseteq K$.

    The in-particular part follows immediately from the correspondence of region representations and graph-decompositions into cliques (\cref{thm:rig-is-gd-into-cliques}).
\end{proof}

Even though one could derive from \cref{thm:GavrilnbhdchordalChar} and \cref{lem:Helly2-acyclic-equiv-3acyclic} that the wheels $W_n$ with $n \geq 4$, and thus also the graphs that contain $W_n$ as an induced subgraph, do not have a $3$-acyclic region representation, we include a basic proof. 

\begin{lemma}\label{lem:wheels-not-3acyclic-RIG}
    A wheel $W_n$ with $n \geq 4$ does not have a $3$-acyclic region representation.
    Moreover, no (possibly infinite) graph $G$ that contains $W_n$ with $n \geq 4$ as an induced subgraph has a $3$-acyclic region representation.
\end{lemma}

\begin{proof}
    Assume $W_n$ has a $3$-acyclic region representation $v \mapsto H_v$ over $H$. 
    Let $W_n$ consist of the cycle $O= x_1x_2 \hdots x_nx_1$ and hub $x$ complete to $O$. 

    Set $H_i \coloneqq H_{x_i} \cap H_x$ for $i = 1, \hdots, n$. Since $v \mapsto H_v$ is 3-acyclic, in particular $2$-acyclic, and $H_{x_i} \cap H_x \neq \emptyset$, it follows that $H_i$ is a tree for $i = 1, \hdots, n$ by \cref{lem:glueingtrees-to-trees}. Next, we claim that $H_i \cap H_j \neq \emptyset$ if and only if $x_ix_j \in E(G)$. If $H_i \cap H_j \neq \emptyset$, then $H_{x_i} \cap H_{x_j} \neq \emptyset$, so $x_ix_j \in E(G)$. Conversely, suppose $x_ix_j \in E(G)$. Then, $x_ix_jx$ is a triangle in $W_n$, so by \cref{lem:bag-cont-clique-in-3-acyclic}, $H_{x_i} \cap H_{x_j} \cap H_x \neq \emptyset$. In particular, $(H_{x_i} \cap H_x) \cap (H_{x_j} \cap H_x) \neq \emptyset$, and thus $H_i \cap H_j \neq \emptyset$. 

    Finally, consider the mapping $x \mapsto H_x$ and $x_i \mapsto H_i$ for $i = 1, \hdots, n$. Now, this is a subtree representation of $W_n$ over the tree $H_x$, contradicting \cref{prop:subtreeintersec-is-chordal}.
    The moreover-part follows immediately from the fact that every ($3$-acyclic) region representation of a graph $G$ restricts to a ($3$-acyclic) region representation of the induced subgraph of $G$. 
\end{proof}

In the remainder of this section, we recall the construction of \gd s\ via coverings from \cite{canonicalGD}, investigate the interplay of cliques and the local covering, and then finally give the construction of the $3$-acyclic \gd s\ that we look for in \cref{mainthm:wheelfree}.

\subsection{Background: Local and normal coverings}\label{subsec:background-local-coverings}

We remind the reader that the graphs we consider in this paper are always simple, i.e.\ they neither have loops nor parallel edges.
A \defn{covering} of a (loopless) graph $G$ is a surjective homomorphism $p \colon \hat G \to G$ such that $p$ restricts to an isomorphism from the edges incident to any given vertex $\hat v$ of $\hat G$ to the edges incident to its projection $v \coloneqq p_r(\hat v)$. 
In this paper, we restrict our view to those coverings $p$ of graphs $G$ whose preimage $p^{-1}(C)$ of any component $C$ of $G$ is connected. Typically, we consider only coverings $p: \hat G \to G$ where both $G$ and $\hat G$ are connected. 

Let $r \in \N$.
A (graph) homomorphism $p \colon \hat G \to G$ is \defn{$r/2$-ball-preserving} if $p$ restricts to an isomorphism from $B_{\hat G}(\hat v, r/2)$ to $B_{G}(v, r/2)$ for every vertex $\hat v$ of $\hat G$ and $v \coloneqq p_r(\hat v)$.
Thus, a surjective homomorphism $p \colon \hat G \to G$ is a covering of a (simple) graph $G$ if and only if $p$ is $2/2$-ball preserving.

In \cite{canonicalGD} the \defn{$r$-local covering $p_r \colon G_r \to G$} of a connected graph $G$ is introduced as the covering of $G$ whose characteristic subgroup is the $r$-local subgroup of the fundamental group of $G$.
Since the formal definition is not relevant to this paper, we refer the reader to \cite[\S 4]{canonicalGD} for details.
In this paper, we will often use that the $r$-local covering preserves the $r/2$-balls.
Indeed, the $r$-local covering is the \defn{universal $r/2$-ball-preserving covering of $G$} \cite[Lemma 4.2, 4.3 \& 4.4]{canonicalGD}, i.e.\ for every $r/2$-ball-preserving covering $p \colon \hat G \to G$ there exists a covering $q \colon G_r \to \hat G$ such that $p_r = p \circ q$.
Following this equivalent description of the $r$-local covering, it is also defined for non-connected graphs $G$.

We refer to the graph $G_r$ as the \defn{$r$-local cover} of $G$.
Note that the $0$-, $1$- and $2$-local covers of a (simple) graph $G$ are forests. For the reader who is familiar with coverings, we remark that the $0$-, $1$- and $2$-local coverings are all the universal covering of $G$.
A \defn{deck transformation} of a covering $p \colon \hat G \to G$ is an automorphism $\gamma$ of $\hat G$ that commutes with $p$, i.e.\ $p = p \circ \gamma$.
We denote the group of deck transformations of a covering $p$ by \defn{$\Gamma(p)$}. A covering $p: \hat G \to G$ is \defn{normal} if for every vertex $v$ of $G$ and every two vertices $\hat v, \hat v' \in p^{-1}(v)$, there is a deck transformation $f \in \Gamma(p)$ such that $f(\hat v) = \hat v'$.

\subsection{Background: Canonical graph-decompositions via coverings}
The main result here is a construction from \cite{canonicalGD} that folds each canonical tree-decomposition of a normal cover $\hat G$ of a graph $G$, such as the $r$-local cover $G_r$, to a graph-decomposition of the graph $G$.

\begin{construction}[\cite{canonicalGD}, Construction~3.8 and Lemma~3.9]
\label{const:TDFoldingToGD}
    Let $p \colon \hat G \to G$ be any normal covering of a (possibly infinite) connected graph $G$.
    Let $(T, \hat \cV)$ be a $\Gamma$-canonical tree-decomposition of $\hat G$, where $\Gamma \coloneqq \Gamma(p)$ is the group of deck transformations of $p$.
    We define the \gd\ $(H,\cV)$ \defn{obtained from $(T,\hat \cV)$ by folding via $p$} as follows.
    The model graph $H$ is the orbit-graph of $T$ under the action of $\Gamma$, i.e.\ $H = T/
    \Gamma$. 
    Denote the quotient map from $T$ to the orbit-graph $H = T / \Gamma$ also by $p$.
    The bag $V_h$ corresponding to a node $h$ of $H$ is $p(\hat V_t)$ for a vertex (equivalently: every vertex) $t$ in the $\Gamma$-orbit $h$.\footnote{In \cite{canonicalGD}, the parts $G_h$ of $(H,\cV)$ are not the induced subgraphs $G[V_h]$ but only the projections $p(\hat G[\hat V_t])$. 
    We will omit this detail, as it will not be relevant in this paper, as these projections will nonetheless be induced whenever one considers decompositions into cliques, due to \cref{lem:CliqueAndLocalCover}.}
    As co-parts $H_v$ of $(H,\cV)$ we choose the projections $p(T_{\hat v})$ of the co-part $T_{\hat v}$ of $\cT$ corresponding to a (equivalently: every) lift $\hat v$ of $v$ to $\hat G$, where we denote also the quotient map from $T$ to the orbit-graph $H = T / \Gamma$ by $p$.

    Consider the equivalent description of the \td\ $(T, \hat\cV)$ as $\hat v \mapsto T_{\hat v}$.
    Then the equivalent description of the \gd\ $(H,\cV)$ is given by $v \mapsto H_v = p(T_{\hat v})$ for a (equivalently: every) lift $\hat v \in p^{-1}(v)$.
\end{construction}

It turns out that certain nice properties of tree-decompositions of $\hat G$ transfer to the graph-decompositions of $G$ obtained by folding via the normal covering $p \colon \hat G \to G$. 
Honest \td s, such as \td s with no empty bags of a connected graph, yield honest \gd s. 
Canonical tree-decompositions yield canonical graph-decompositions: 

\begin{lemma}[\cite{canonicalGD}, Lemma 3.12] \label{lem:gdobtainedviafoldings-canonical}
    Let $\cH$ be a \gd\ of a (possibly infinite) graph $G$ obtained from a $\Gamma$-canonical \td\ $\cT$ of $\hat G$ by folding via any normal covering $p\colon \hat G \to G$. Then $\cH$ is canonical if $\cT$ is canonical.
\end{lemma}

\subsection{Interplay: Local covering \texorpdfstring{\&}{and} cliques}

Let us collect some lemmas regarding the interplay of cliques and the local coverings. 
Local covers of graphs interact well with cliques, made formal in the following lemma:

\begin{lemma}[\cite{computelocalSeps}, Lemma~5.15]\label{lem:CliqueAndLocalCover}
    Let $G$ be a (possibly infinite) graph and let $r \geq 3$ an integer.
    \begin{enumerate}
        \item For every clique $\hat X$ of $G_r$, its projection $p_r(\hat X)$ is a clique of $G$, and $p_r$ restricts to a bijection from $\hat X$ to $p_r(\hat X)$.
        \item For every clique $X$ of $G$, there exists a clique $\hat X$ of $G_r$ such that $p_r$ restricts to a bijection from $\hat X$ to $X$.
        \item\label{item:CliquesMove} $N_{G_r}[\hat X] \cap N_{G_r}[\gamma(\hat X)] = \emptyset$ for every clique $\hat X$ of $G_r$ and every $\gamma \in \Gamma(p_r) \setminus \{\id_{G_r}\}$.
    \end{enumerate}
\end{lemma}

For a given clique $X$ of $G$, we refer to all cliques $\hat X$ of $G_r$ such that $p_r$ restricts to a bijection from $\hat X$ to $X$ as \defn{lifts of the clique $X$ to $G_r$}.
\cref{lem:CliqueAndLocalCover} ensures that every clique $\hat X$ of $G_r$ projects to a clique $X := p_r(\hat X)$ of $G$ and $\hat X$ is indeed a lift of the clique $X$ in $G$, every clique of $G$ has a lift to $G_r$, and not only are two distinct lifts of any clique disjoint, but their closed neighborhoods are also disjoint.

The intersection graph $\bK(G_r)$ of the maximal cliques of $G_r$ covers the intersection graph $\bK(G)$ of the maximal cliques of $G$ via $p_r$:

\begin{lemma}\label{lem:clique-graph-covers-clique-graph}
    Let $G$ be a (possibly infinite) graph and $r \geq 3$ an integer.
    Then $p_r$ induces a covering map from $\bK(G_r)$ to $\bK(G)$ via $\hat X \mapsto p_r(\hat X)$ which restricts to an isomorphism from $\bK(G_r)[K_{G_r}(\hat X)]$ to $\bK(G)[K_G(X)]$ for every clique $X$ of $G$ and each of its lifts $\hat X$ to $G_r$.
    In particular, the maximal cliques of $G_r$ are precisely the lifts of the maximal cliques of $G$.
\end{lemma}

\noindent In the context of \cref{lem:clique-graph-covers-clique-graph}, we denote the covering map from $\bK(G_r)$ to $\bK(G)$ induced by $p_r$ again by $p_r$.

\begin{proof}
    First, we show that $V(\bK(G_r))$ is the set $K(G)$ of lifts of the maximal cliques of $G$.
    Let $\hat X$ be a maximal clique of $G_r$.
    Then \cref{lem:CliqueAndLocalCover}~(i) ensures that $X \coloneqq p_r(\hat X)$ is a clique of $G$.
    Let $Y$ be a maximal clique of $G$ containing $X$.
    By \cref{lem:CliqueAndLocalCover}~(ii), there is a lift $\hat Y$ of $Y$ to $G_r$.
    By potentially shifting $\hat Y$ with a deck transformation of $p_r$, we assume that $\hat Y$ meets $\hat X$.
    Since distinct lifts of the clique $X$ are disjoint by \cref{lem:CliqueAndLocalCover}~(iii), $\hat Y$ must contain the whole lift $\hat X$ of $X$.
    In particular, since $\hat X$ is a maximal clique, $\hat Y = \hat X$, and thus $X = Y$ is a maximal clique of $G$.
    We can similarly prove that every maximal clique of $G_r$ projects to a maximal clique of $G$.
    All in all, this shows that the maximal cliques of $G_r$ are precisely the lifts of the maximal cliques of $G$ and that $p_r$ induces a surjection from $V(\bK(G_r))$ to $V(\bK(G))$.

    \cref{lem:CliqueAndLocalCover}~(iii) ensures that, for every clique $X$ and each of its lifts $\hat X$ to $G_r$, the map induced by $p_r$ restricts to a bijection from $K_{G_r}(\hat X)$ to $K_G(X)$. Moreover, for every two cliques $Y,Z \in K_G(X)$ and their respective lifts $\hat Y, \hat Z \in K_{G_r}(\hat X)$, the map $p_r$ restricts to a bijection from $\hat Y \cap \hat Z$ to $Y \cap Z$.
    Thus, $p_r$ induces an isomorphism from $\bK(G_r)[K_{G_r}(\hat X)]$ to $\bK(G)[K_G(X)]$.
    In particular, $p_r$ induces a covering map from $\bK(G_r)$ to $\bK(G)$.
\end{proof}

\subsection{Graph-decomposition into cliques \& region representation via coverings}

Now we investigate the \gd\ of a graph $G$ obtained by folding a \td\ of its $r$-local cover $G_r$ into cliques via $p_r$.

First, recall that by \cref{thm:chordal}, if $G_r$ is chordal, then $G_r$ admits a canonical tree-decomposition into cliques. If we require only a $\Gamma(p_r)$-canonical tree-decomposition, we can even obtain a tree-decomposition into \emph{maximal} cliques:

\begin{theorem}[\cite{canTDofChordalGraphs}, Theorem 2 ]\label{thm:chordalTDmaximal}
    Let $G$ be a connected, locally finite graph and $p \colon \hat G \to G$ a normal covering.
    If $\hat G$ is chordal, then $\hat G$ admits a $\Gamma(p)$-canonical \td\ into its maximal cliques (equivalently: $\hat G$ has a $\Gamma(p)$-canonical clique tree).
\end{theorem}

Given a graph $G$, tree-decompositions of its local cover $G_r$ into cliques yield graph-decompositions of $G$ into cliques. 

\begin{theorem}[\cite{canTDofChordalGraphs}, Lemma 4.3]\label{lem:region-rep-obtained-via-folding}
    Let $G$ be a (possibly infinite) graph and $r \geq 3$ an integer.
    Suppose that $\cT$ is
    a $\Gamma(p_r)$-canonical \td\ of $G_r$ and let $\cH$ be the \gd\ $\cH$ obtained from $\cT$ by folding via $p_r$.
    Then:
    \begin{enumerate}
        \item $\cT$ is into cliques if and only if $\cH$ is into cliques.
        \item $\cT$ is into maximal cliques if and only if $\cH$ is into maximal cliques.
    \end{enumerate}
\end{theorem}

\noindent In the context of \cref{lem:region-rep-obtained-via-folding}, we now refer to the map $v \mapsto H_v$ corresponding to a \gd\ $\cH$ into cliques as the region representation of $G$ \defn{obtained from the subtree representation $\hat v \mapsto T_{\hat v}$ of $G_r$} corresponding to the \td\ $\cT$ into cliques
\defn{by folding via $p_r$}.

\begin{lemma}
\label{cor:pr-isomorphism-between-cliques}
    Let $G$ be a (possibly infinite) connected graph and $r \geq 3$ an integer. 
    Let $\hat v \mapsto T_{\hat v}$ be a $\Gamma(p_r)$-canonical subtree representation of $G_r$ over a tree $T$ and let $v \mapsto H_v$ be the region representation of $G$ over $H = T/ \Gamma(p_r)$ obtained from $\hat v \mapsto T_{\hat v}$ by folding via $p_r$.
    Then, for every clique $X$ of $G$ and lift $\hat X$ of $X$ to $G_r$, the map $p_r$ restricts to an isomorphism from $T_{\hat X}$ to $H_X$.
\end{lemma}
\begin{proof}
Fix a clique $X$ of $G$ and a lift $\hat X$ of $X$ to $G_r$ (which exists by \cref{lem:CliqueAndLocalCover}). Since $H_X = \bigcup_{\hat x \in \hat X} p_r(T_{\hat x})$, the map $p_r$ restricts to a surjection from $T_{\hat X}$ to $H_X$. Next, we show that $p_r$ restricts to an injection from $T_{\hat X}$ to $H_X$. First, we need the following observation. Let $t$ be a node of $T$ and let $V_t \coloneqq \{\hat v \in V(G_r) \mid t \in V(T_{\hat v})\}$ be the bag assigned to $t$. Since $\hat v \mapsto T_{\hat v}$ is a subtree representation, it follows that $V_t$ is a clique of $G_r$. Therefore, for every $\hat v \in V_t$, it holds that $V_t \subseteq N_{G_r}[\hat v]$.

By the above observation, $\bigcup_{t \in V(T_{\hat x})} V_t \subseteq N_{G_r}[\hat x]$ for every $\hat x \in \hat X$, and thus $\bigcup_{t \in V(T_{\hat X})} V_t \subseteq N_{G_r}[\hat X]$. 
Similarly, for all $\gamma \in \Gamma(p_r)$, it holds that $\bigcup_{t \in V(T_{\gamma(\hat X)})} V_t \subseteq N_{G_r}[\gamma(\hat X)]$. By \cref{lem:CliqueAndLocalCover}~\cref{item:CliquesMove}, $N_{G_r}[\hat X]$ is disjoint from $N_{G_r}[\gamma(\hat X)]$ for all $\gamma \in \Gamma(p_r)$ distinct from the identity. Therefore, $T_{\hat X}$ is disjoint from $T_{\gamma(\hat X)}$ for all $\gamma \in \Gamma(p_r)$ distinct from the identity. 
We conclude that $p_r(T_{\hat X})$ is isomorphic to $T_{\hat X}$, and so $p_r$ restricts to an isomorphism from $T_{\hat X}$ to $H_X$.
\end{proof}

We can now show that decompositions obtained from $G_3$ via folding by $p_3$ are 3-acyclic: 
\begin{lemma}\label{clm:3-acyclic}
    Let $G$ be a (possibly infinite) connected graph and $r \geq 3$ an integer.
    Suppose that $\hat v \mapsto T_{\hat v}$ is a $\Gamma(p_r)$-canonical subtree representation of $G_r$ over $T$.
    Let $v \mapsto H_v$ be the region representation of $G$ obtained from $\hat v \mapsto T_{\hat v}$ by folding via $p_r$. Then, $v \mapsto H_v$ is 3-acyclic. 
\end{lemma}

\begin{proof}
    Suppose for a contradiction that there is a set $X \subseteq V(G)$ of size at most three such that $H_X$ contains a cycle. By \cref{lem:2acyclic-yields-Xcontainscycle-size3-is-clique}, $X$ is a clique. Fix some lift $\hat X$ of $X$ to $G_r$. Since $p_r$ restricts to an isomorphism from $T_{\hat X}$ to $H_X$ by \cref{cor:pr-isomorphism-between-cliques} and $T_{\hat X}$ is a tree, it follows that $H_X$ is a tree, a contradiction. 
\end{proof}

\noindent In the context of \cref{clm:3-acyclic}, the $H_v$ are the only valid choice for co-parts of the corresponding \gd\ $(H,\cV)$ from \cref{const:TDFoldingToGD}, as the $H_v$ are induced subgraphs of $H$. 
More generally, there is no need to specify the family of co-parts whenever a \gd\ is $2$-acyclic and honest, as then there is only one valid choice:

\begin{lemma}\label{lemma:HvAreInduced}
    If $(H,\cV)$ is a $2$-acyclic  \gd\ of a graph $G$ with a fixed family $(H_v)_{v \in G}$ of co-parts, then the subgraphs $H_{V(G)}[W_v]$ induced on the co-bags are trees.
    Specifically, $H_v = H_{V(G)}[W_v]$.

    In particular, if $v \mapsto H_v$ is a $2$-acyclic region representation of $G$ over $H$ with $H = \bigcup_{v \in G} H_v$, then the $H_v$ are induced subgraphs of $H$.
\end{lemma}

\begin{proof}
    The $H_v$ are spanning trees of $H_{V(G)}[W_v]$.
    So suppose for a contradiction that there is some edge $f \notin H_v$ of $H_{V(G)}$ with both endvertices in $H_v$.
    Since $H_{V(G)} = \bigcup_{v \in G} H_v$, there is some vertex $u \in G$ such that $f \in H_u$.
    As $H_v$ is connected, $H_v + f \subseteq H_v \cup H_u$ contains a cycle, which contradicts that $v \mapsto H_v$ is $2$-acyclic.
\end{proof}

\subsection{Proof of \texorpdfstring{\cref{mainthm:wheelfree}}{Theorem XX}}
We are now ready to prove \cref{mainthm:wheelfree}. 
\begin{proof}[Proof of \cref{mainthm:wheelfree}]
Clearly, \cref{wheelfree:ii} implies \cref{wheelfree:iv} and \cref{wheelfree:iii} implies \cref{wheelfree:iv}. By \cref{lem:wheels-not-3acyclic-RIG},
        \cref{wheelfree:iv} implies \cref{wheelfree:i}.  
   
    \begin{claim}
        \cref{wheelfree:i} implies \cref{wheelfree:ii}. 
    \end{claim}
    \begin{claimproof}
     By \cref{BasicCharacterization}, the $3$-local cover $G_3$ of $G$ is chordal, so $G_3$ has a $\Gamma(p_3)$-canonical clique tree $T$ by \cref{thm:chordalTDmaximal}.
     By \cref{lem:clique-graph-covers-clique-graph}, $p_3$ induces a covering map from $\bK(G_3)$ to $\bK(G)$ via $\hat X \mapsto p_3(\hat X)$ for each node $\hat X$ of $\bK(G_3)$.
        Since $T$ is a spanning subgraph of $\bK(G_3)$, its projection $H \coloneqq p_3(T)$ is a spanning subgraph of $\bK(G)$. We claim that this graph $H$ is a clique graph of $G$. Since $H$ is a spanning subgraph of $\bK(G)$, it remains to show that $H[K_G(v)]$ is connected for each $v \in V(G)$. Fix a vertex $v$ of $G$ and consider $H_v = H[K_G(v)]$. By \cref{lem:clique-graph-covers-clique-graph}, the covering map from $\bK(G_3)$ to $\bK(G)$ induced by $p_3$ restricts to an isomorphism from $T[K_{G_3}(\hat v)]$ to $H_v$ for each lift $\hat v$ of $v$ to $G_3$. Since $T$ is a clique tree of $G_3$, it follows that $T[K_{G_3}(\hat v)]$ is connected, so $H_v$ is also connected. This proves that $H$ is a clique graph of $G$. By \cref{clm:3-acyclic}, $H$ is 3-acyclic. \end{claimproof}

\begin{claim}
    \cref{wheelfree:i} implies \cref{wheelfree:iii}. 
\end{claim}
 \begin{claimproof}

    By \cref{BasicCharacterization}, the $3$-local cover $G_3$ of $G$ is chordal. By \cref{thm:chordal}, there is a canonical subtree representation $\hat v \mapsto T_{\hat v}$ of $G_3$ over a tree $T$. By \cref{lem:region-rep-obtained-via-folding}, we now consider the region representation $v \mapsto H_v$ of $G$ over the orbit-graph $H = T/ \Gamma(p_3)$ obtained from $\hat v \mapsto T_{\hat v}$ by folding via $p_3$. By \cref{lem:gdobtainedviafoldings-canonical} and the equivalence between \gd s and region representations, since $\hat v \mapsto T_{\hat v}$ is canonical, it follows that $v \mapsto H_v$ is canonical. Finally, $v \mapsto H_v$ is 3-acyclic by \cref{clm:3-acyclic}.
    \end{claimproof}
This completes the proof.
\end{proof}

\section{Characterizations via 3-acyclic region representations}\label{sec:3acyclic-racyclic}

In this section, we prove \cref{lem:char-locallychordal-via-3-acyclic-region-rep}, which we restate here. 

\CharViaAcyclicRegionRep*

First, we prove some preparatory lemmas, which enable us to prove \cref{lem:char-chordal-via-3-acyclic-region-rep}, the statement analogous to \cref{lem:char-locallychordal-via-3-acyclic-region-rep} for chordal graphs.

\begin{theorem}\label{lem:char-chordal-via-3-acyclic-region-rep}
    Let $v \mapsto H_v$ be a $3$-acyclic region representation of a (possibly infinite) graph $G$ over a graph $H$.
    Then $G$ is chordal if and only if the region representation $v \mapsto H_v$ is acyclic.
\end{theorem}

Then we deduce from \cref{lem:char-chordal-via-3-acyclic-region-rep} the following structural version of \cref{lem:char-locallychordal-via-3-acyclic-region-rep}:

\begin{theorem}\label{lem:inducedcycles-vs-cyclesinHX}
    Let $v \mapsto H_v$ be a $3$-acyclic region representation of a (possibly infinite) graph $G$ over a graph $H$, and let $X$ be a vertex-subset of $G$.
    Then $G[X]$ is an induced cycle of length at least $4$ if and only if $X$ is inclusion-minimal such that $H_X$ contains a cycle.
\end{theorem}

\subsection{Interplay: Helly, acyclic\texorpdfstring{ \&}{, and} clique graphs}

It is well-known\footnote{A proof is given in \cite[Lemma~1]{GavrilChordalGraphsSTIG}.} that families of subtrees of a tree are Helly. Thus, subtree representations and clique trees are Helly.

\begin{theorem}[Folklore]\label{thm:subtree-reps-are-helly}
    Every family of subtrees of a tree is Helly.
    In particular, subtree representations are Helly.
\end{theorem}

\noindent
We remark that we restricted the notion of an infinite family $\bH$ of subgraphs of a graph being \emph{Helly} to just its finite subsets, as there are infinite families of subtrees of a tree that pairwise intersect but have an empty intersection:

\begin{example}
    The final segments of a fixed ray form an infinite family of subtrees of a tree whose elements pairwise intersect, but whose intersection is empty.
\end{example}

\subsection{Background: Minimal separators and separations}

Let $u,v$ be two vertices and $X$ a vertex-subset of a graph $G$.
We say that $X$ \defn{separates $u$ and $v$} in $G$, and call $X$ an \defn{$u$--$v$ separator} of $G$, if $X \subseteq V(G) \setminus \{u,v\}$ and every $u$--$v$ path in $G$ meets $X$.
If $X$ separates some two vertices in $G$, then~$X$ is a \defn{separator} of $G$.
A separator $X$ is a \defn{minimal (vertex) separator} of $G$ if there are two vertices $u, v$ of $G$ such that $X$ is inclusion-minimal among the $u$--$v$ separators in $G$. 

Dirac characterized chordal graphs by the structure of their minimal separators: 

\begin{restatable}[Dirac~\cite{dirac1961rigid}, {Theorem~1}]{theorem}{tightsepchar}
    \label{thm:CharacterisationChordalViaMinimalVertexSeparator} 
    A (possibly infinite)\footnote{It is immediate from the proof presented in \cite[Theorem~2.1]{blair1993introduction} that Dirac's theorem also holds for infinite graphs.} graph is chordal if and only if  all its minimal separators are cliques.
\end{restatable} 

A pair $\{A,B\}$ of vertex-subsets of a graph $G$ $G$ is a \defn{separation} of $G$ if $A \cup B = V(G)$ and there is no edge between $A \setminus B$ and $B \setminus A$.
Its \defn{separator} is $A \cap B$, its \defn{sides} are $A$ and $B$, and its \defn{proper sides} are $A \setminus B$ and $B \setminus A$.

\subsection{Proof of \texorpdfstring{\cref{lem:char-locallychordal-via-3-acyclic-region-rep}}{Theorem XX}}

Now we prove the main results of this section. 

\begin{proof}[Proof of \cref{lem:char-chordal-via-3-acyclic-region-rep}]
    For the sake of this proof, we may assume $H = H_{V(G)}$ by disregarding all edges which are not in $H_v$ for some vertex $v$ of $G$.
    Then the components $G'$ of $G$ are in a $1$-to-$1$ correspondence with the components of $H$ via $G' \mapsto H_{V(G')}$.
    Since a graph is chordal if and only if each of its components is chordal, we may assume for this proof that $G$, and thus $H$, is connected.

    If $H$ is acyclic (and thus a tree), then $v \mapsto H_v$ is a subtree representation of $G$ over a tree $H$, and thus by \cref{prop:subtreeintersec-is-chordal}, $G$ is chordal.
    To prove the converse, let us assume that $G$ is chordal.
    As $G$ is connected and $v \mapsto H_v$ represents $G$, the graph $H$ is connected, as well.
    It remains to show that $H$ is acyclic. Suppose for a contradiction that $C$ is a cycle in $H$.
    By choosing for each edge $f$ of the cycle $C$ a vertex $x_f$ of $G$ such that $f$ is an edge of $H_{x_f}$, we find a finite set $X = \{x_f \mid f \in E(C)\}$ such that $H_X$ contains a cycle.
    Moreover, as $v \mapsto H_v$ represents $G$, the subgraph $G[X]$ is connected.
    Thus, it suffices that $H_Y$ is a tree for every finite vertex-subset $Y$ of $G$ such that $G[Y]$ is connected 
    We proceed by induction on the number of maximal cliques of the finite chordal graph $G[Y]$.

    The base case follows immediately from \cref{lem:bag-cont-clique-in-3-acyclic} together with \cref{lem:gavril-lemmas}~\cref{gavril:1}.
    So we may assume that $G[Y]$ is not complete, i.e.\ there are two non-adjacent vertices $u,v \in Y$.
    In particular, $u,v$ are separated by $Y \setminus \{u,v\}$ in $G[Y]$.
    Pick an inclusion-minimal $u$--$v$ separator $S$ in $G[Y]$.
    Fix a separation $\{A,B\}$ of $G[Y]$ with separator $A \cap B = S$ such that $u \in A$ and $v \in B$.
    As every maximal clique $K$ of $G[Y]$ satisfies either $K \subseteq A$ or $K \subseteq B$, both subgraphs $G[A], G[B]$ induced on the sides of $\{A,B\}$ have fewer maximal cliques than $G[Y]$.
    Since $G[Y]$ is chordal, $S$ is a clique by \cref{thm:CharacterisationChordalViaMinimalVertexSeparator}.
    Thus, the subgraphs $G[A],G[B]$ are connected.
    As induced subgraphs of a chordal graph, both $G[A],G[B]$ are chordal.
    The induction assumption applied to $A$ and $B$ yields that both $H_A$ and $H_B$ are trees.
    Since $\{A,B\}$ is a separation and $v \mapsto H_v$ represents $G$, their intersection $H_A \cap H_B$ is precisely $H_S$, which is also a tree by the induction assumption.
    Thus, $H_Y = H_A \cup H_B$ is a tree by \cref{lem:glueingtrees-to-trees}.
\end{proof}

Now, we prove the converse of \cref{lem:inducedcycles-vs-cyclesinHX} in slightly more detail.

\begin{lemma}\label{lem:converseofinducedcycle-yields-cycleHX}
    Let $v \mapsto H_v$ be a $2$-acyclic region representation of a (possibly infinite) graph $G$ over a graph $H$.
    If $X$ is an inclusion-minimal vertex-subset of $G$ such that $H_X$ contains a cycle, then $G[X]$ is an induced cycle.
\end{lemma}

\begin{proof}
    Assume that $X$ is inclusion-minimal so that $H_X$ contains a cycle.
    Since $v \mapsto H_v$ is $2$-acyclic, $X$ has size at least $3$.
    If $X$ has size $3$, then $G[X]$ is an induced cycle by \cref{lem:2acyclic-yields-Xcontainscycle-size3-is-clique}.
    Thus, we assume that $X$ has size at least $4$.
    As $X$ is inclusion-minimal with the property that $H_X$ contains a cycle and $X$ has size at least $4$, the restriction of $v \mapsto H_v$ to $X$ is a $3$-acyclic region representation of $G[X]$.
    Therefore, by applying \cref{lem:char-chordal-via-3-acyclic-region-rep} to $G[Y]$ with $Y \subseteq X$ and its $3$-acyclic region representation, the restriction of $v \mapsto H_v$ to $Y$, we deduce that $X$ is inclusion-minimal such that $G[X]$ is not chordal, i.e.\ $G[X]$ is an induced cycle of length at least $4$, as desired.
\end{proof}

\begin{proof}[Proof of \cref{lem:inducedcycles-vs-cyclesinHX}]
    The backwards direction follows immediately from \cref{lem:converseofinducedcycle-yields-cycleHX}.
    For the forwards direction, assume that $G[X]$ is an induced cycle of length at least $4$. 
    In particular, $G[X]$ is connected but not chordal.
    Thus, \cref{lem:char-chordal-via-3-acyclic-region-rep} applied to $v \mapsto H_v$ restricted to $X$, which is a $3$-acyclic region representation of $G[X]$, yields that $H_X$ is not acyclic, i.e.\ $H_X$ contains a cycle.
\end{proof}

\begin{proof}[Proof of \cref{lem:char-locallychordal-via-3-acyclic-region-rep}]
    Since $G$ admits a $3$-acyclic region representation $v \mapsto H_v$, no wheel $W_n$ is an induced subgraph of $G$ for $n \geq 4$ by \cref{lem:wheels-not-3acyclic-RIG}.
    \cref{lem:inducedcycles-vs-cyclesinHX} yields that $G$ has no induced cycles of length at least $4$ and at most $r$ if and only if $v \mapsto H_v$ is $r$-acyclic.
    The former is, for wheel-free graphs $G$, equivalent to being $r$-locally chordal by \cref{BasicCharacterization}, as desired.
\end{proof}

\section{Characterization via graph-decompositions that 
locally induce tree-decompositions}\label{sec:CharViaGDLocDerTD}

By the definition, each $r/2$-ball $B_{r/2}(v)$  of an $r$-locally chordal graph is chordal, and thus admits a tree-decomposition $(T^v, \cV^v)$ into cliques (\cref{thm:chordal}).
In this section, we revisit the statement from the introduction that $r$-locally chordal graphs admit decompositions which look like a single decomposition formed by ``putting together'' the tree-decompositions of the $r/2$-balls. Specifically, we show that $r$-acyclic graph-decompositions of an $r$-locally chordal graph form a common description of some choice of tree-decompositions $(T^v, \cV^v)$ of the $r/2$-balls into cliques (\cref{main:InducingTDsonBalls} and \cref{thm:locally-derived-TDs}).

First, we formalize what the locally induced decompositions are.
Let $(H,\cV)$ be a decomposition of a graph $G$.
The decompositions which $(H,\cV)$ \defn{$r$-locally induces in $G$} are the families $(H^{v,r}, \cV^{v,r})$ where $H^{v,r} := \bigcup \{H_v \mid d_G(u,v) < r/2\}$ and $V^{v,r}_h := V_h \cap N^{r/2}[v]$ for each node $h$ of $H^{v,r}$.
For a given vertex $v$ of $G$, we assign a vertex $u \in N^{r/2}[v]$ the co-part $H^{v,r}_u$ of $u$ in the decomposition $(H^{v,r}, \cV^{v,r})$, where $H^{v,r}_u := H_u$ if $d_G(u,v) < r/2$ and $H^{v,r}_u := H_u \cap H^{v,r}$ if $d_G(u,v) = r/2$.
We can now state the main result of this section. 

\begin{mainresult}\label{main:InducingTDsonBalls}
Let $G$ be a locally finite graph and $r \geq 3$ an integer. The following are equivalent: 
\begin{enumerate}
     \item \label{item:inducingTD1} $G$ is $r$-locally chordal. 
     \item \label{item:inducingTD2} $G$ admits a graph-decomposition into its maximal cliques which $r$-locally induces tree-decompositions. 
     \item $G$ admits a canonical graph-decomposition into cliques which $r$-locally induces tree-decompositions. 
     \item $G$ admits a graph-decomposition into cliques which $r$-locally induces tree-decompositions. 
\end{enumerate}
\end{mainresult}

In general, the $r$-locally induced families $(H^{v, r},\cV^{v, r})$ satisfy \cref{axiomH1}.
For odd $r$, they even satisfy \cref{axiomH2}, and thus are graph-decompositions.
But, for even $r$, the subgraphs of $H^{v,r}$ induced on the co-bag of a vertex $u \in N^{r/2}(v)$ are not necessarily connected.
Thus, the decomposition $(H^{v,r},\cV^{v,r})$ does not in general satisfy \cref{axiomH2}, and hence need not be graph-decompositions.

We show that every $r$-acyclic decomposition $(H,\cV)$ $r$-locally induces tree-decompositions. In fact, something even stronger is true. A decomposition $(H, \cV)$ is \defn{connected-$r$-acyclic} if, for every connected vertex-subset $X$ of $G$ of at most $r$ vertices, the union $H_X$ is acyclic. 
From the definition, every $r$-acyclic \gd\ $(H,\cV)$ is connected-$r$-acyclic. If $(H,\cV)$ is into cliques, then the converse holds as well:

\begin{lemma}\label{lem:ConAcyclicEquivToAcyclicForIntoCliques}
    Let $(H,\cV)$ be a \gd\ of a (possibly infinite) graph $G$ into cliques, and let $r \geq 1$ be an integer.
    Then $(H,\cV)$ is connected-$r$-acyclic if and only if $(H,\cV)$ is $r$-acyclic.
\end{lemma}

\begin{proof}
    It follows from the definition that every $r$-acyclic decomposition is connected-$r$-acyclic.

    Conversely, assume that $(H, \cV)$ is connected-$r$-acyclic. Suppose that $X$ is a vertex-subset of $G$ such that $H_X$ contains a cycle $O$.
    By \cref{lem:strongerH2forRegRep}, the connected subgraph $O$ of $H$ is contained in $H_{X'}$ for some component $G[X']$ of $G[X]$ with $X' \subseteq X$. Therefore, $(H, \cV)$ is $r$-acyclic if it is connected-$r$-acyclic. 
\end{proof}

It turns out that connected-$r$-acyclic decompositions are precisely the ones that $r$-locally induce tree-decompositions:

\begin{theorem}\label{thm:connectedAcyclicEquivLocallyInducingTDs}
    Let $(H,\cV)$ be a graph-decomposition of a (possibly infinite) graph $G$ and let $r \geq 1$ be an integer.
    Then $(H,\cV)$ is connected-$r$-acyclic if and only if the decompositions it $r$-locally induces are tree-decompositions.
\end{theorem}

First, we characterize those graph-decomposition whose locally induced decompositions are tree-decompositions:

\begin{lemma}\label{lemma:CharacterisingLocallyInducingTDs}
    Let $(H,\cV)$ be a graph-decomposition of a (possibly infinite) graph $G$, and let $r \geq 1$ be an integer.
    Then the decompositions that are $r$-locally induced by $(H,\cV)$ are tree-decompositions if and only if
    \begin{itemize}
        \item $r$ is odd and $H_{N^{r/2}[v]}$ is a tree, or
        \item $r$ is even and $H_{N^{r/2-1}[v] \cup \{u\}}$ is a tree for every $u \in N^{r/2}(v)$.
    \end{itemize}
\end{lemma}

\begin{proof}
    Let $v$ be a vertex of $G$, set $B := B_{r/2}(v)$, and let $(T,\cX) := (H^{v,r}, \cV^{v,r})$ be the decomposition $(H,\cV)$ $r$-locally induces.
    We warn the reader that we use terms such as \emph{bag}, \emph{part} and \emph{co-part} in the context of the pair $(T,\cX)$ before knowing that $(T,\cX)$ is a \gd\ of $B$.
    Let us make the following observations about the definition of $(T,\cX)$, without assuming anything about $(H,\cV)$ or $(T,\cX)$.

    Since \cref{axiomH1} holds for $(H,\cV)$, \cref{axiomH1} also holds for the pair $(T,\cX)$ with respect to $B$.
    As $(H,\cV)$ $r$-locally induces $(T,\cX)$,
    \begin{itemize}
        \item if $r$ is odd, then the co-part $T_u$ of a vertex $u \in B_{r/2}(v)$ in $(T,\cX)$ is the co-part $H_u$ of $u$ in $(H,\cV)$, and
        \item if $r$ is even, then the co-part $T_u$ of a vertex $u \in B_{r/2}(v)$ in $(T,\cX)$ is either $H_u$ if $u \in N^{r/2-1}[v]$ or $H_u \cap T$ if $u \in N^{r/2}(v)$.
    \end{itemize}
    Hence, since \cref{axiomH2} holds for $(H,\cV)$, to show \cref{axiomH2} for $(T,\cX)$ it suffices to prove that $H_u \cap T$ is connected if $r$ is even and $u \in N^{r/2}(v)$.
    So let us assume that $r$ is even; in particular, $T = H_{N^{r/2 -1}[v]}$.

    No matter which of the two assumptions about the decomposition $(H,\cV)$ we make, the graph $T = H_{N^{r/2 -1}[v]}$ is a tree and every $H_u$ is a tree for every vertex $u$ of $G$.
    Thus, consider $u \in N^{r/2}(v)$. By \cref{lem:glueingtrees-to-trees}, $H_u \cup T$ is a tree if and only if $H_u \cap T$ is connected.
    The latter is precisely \cref{axiomH2} of $(T,\cX)$.
    The former is the property described in the statement of \cref{lemma:CharacterisingLocallyInducingTDs}.
    Hence, the desired statement holds.
\end{proof}

Before we prove our next lemma, we give some notation around paths. For a path $P$, we denote the number of vertices of $P$ by \defn{$|P|$} and the number of edges of $P$ by \defn{$||P||$}. The \defn{length of $P$} is the number of its edges, so if $P$ has length $\ell$ then $|P| = \ell + 1$. For a fixed vertex $v$ of $P$, we denote by \defn{$Pv$} the subpath of $P$ starting at the first vertex of $P$ and ending at $v$, and by \defn{$vP$} the subpath of $P$ starting at $v$ and ending at the last vertex of $P$. 

\begin{lemma}\label{lem:centralvertex}
    Let $G$ be a (possibly infinite) graph, let $r \geq 1$ be an integer, and let $U$ be a connected set of $r$ vertices of $G$.
    Then some vertex $v \in U$ has distance at most $r/2$ from every $u \in U$ and at most one $u \in U$ has distance precisely $r/2$ from $v$.
    In particular, if $r$ is odd, then some vertex $v \in U$ has distance less than $r/2$ from every $u \in U$.
\end{lemma}

\begin{proof}
    Let $P$ be a longest induced path in the induced subgraph $G[U]$.
    Let $v$ be one of the `central' vertices of $P$:
    If $||P||$ is even, then let $v$ be the vertex such that $||vP|| = ||Pv||$.
    Otherwise, $||P||$ is odd, and we let $v$ be the vertex such that $||vP|| -1 = ||Pv||$.
    In both cases, $\ell := ||vP|| \geq ||Pv||$.
    
    Let $T$ be a spanning tree of $G[U]$ containing $P$.
    If some inclusion-maximal path $Q$ in $T$ that ends in $v$ is longer than $\ell$, then at least one of $QvP$ or $PvQ$ is a path which is longer than $P$.
    Since this contradicts the choice of $P$, we may assume that all leaves of $T$ have distance at most $\ell$ from $v$ in $T$.
    Hence, if $\ell < r/2$, then the tree $T$ witnesses that every $u \in U$ has distance at most $\ell < r/2$ from $v$, as desired.
    Thus, we assume $||vP|| = \ell = r/2$; in particular, $r$ is even, since $\ell$ is an integer.
    Since $P \subset G[U]$, it follows that $P$ has at most $r$ vertices, so $||P|| \leq r-1 = r/2 + (r/2-1)$.
    Hence, the choice of $v$ yields $||Pv|| = (r/2 - 1)$, so $V(P) = U$.
    Thus, the path $P = T$ shows that the statement holds.

    The in-particular part follows from the fact that distances in graphs are always integers.
\end{proof}

\begin{proof}[Proof of \cref{thm:connectedAcyclicEquivLocallyInducingTDs}]
    Assume that the decompositions $r$-locally induced by $(H,\cV)$ are tree-decompositions.
    We claim that $(H,\cV)$ is connected-$r$-acyclic.
    Indeed, let $U$ be a connected set of at most $r$ vertices of $G$.
    Since $U$ is connected in $G$, \cref{lem:centralvertex} yields some vertex $v \in U$ that has distance at most $r/2$ from every $u \in U$ but at most one of them has distance precisely $r/2$.
    Hence, $H_U$ is acyclic by \cref{lemma:CharacterisingLocallyInducingTDs}.

    To prove the converse, assume that $(H,\cV)$ is connected-$r$-acyclic.
    By induction on $r \geq 1$, we show that the decompositions $r$-locally induced by $(H,\cV)$ are tree-decompositions by checking the equivalent properties described in \cref{lemma:CharacterisingLocallyInducingTDs}.
    Let $v$ be a  vertex of $G$.
    The base case $r = 1$ follows immediately from $(H,\cV)$ being connected-$1$-acyclic, as $N^{1/2}[v] = \{v\}$ is connected.
    For the induction step, we assume that $r \geq 2$ and the decompositions $(r-1)$-locally induced by $(H,\cG)$ are tree-decompositions.

    Set $d := \lfloor r/2 \rfloor$, and $d' := d-1$.
    Set $N' := N^{d'}[v]$.
    By the induction assumption, $H_{N'}$ is a tree.
    If $r$ is even, then we set $W := \emptyset$, $N^\ast := N'$, and $H^\ast := H_{N^\ast}$.
    If $r$ is odd, then let $W$ be any subset of $N^d(v)$ such that, for $N^\ast := N' \cup W$, the union $H^\ast := H_{N^\ast}$ is a tree.
    Let $u \in N^d(v) \setminus W$ be arbitrary.
    
    To conclude the induction step, it suffices to show that $H^{\ast} \cup H_u$ is acyclic, and thus a tree.
    Suppose for a contradiction that $H^{\ast} \cup H_u$ contains a cycle, i.e.\ $H_u$ contains an $H^*$-path $Q$ with ends~$h_1,h_2$.
    Since $u \in N^d(v)$, there is $u' \in N'$ so that $uu'$ is an edge of $G$, so $H_u \cup H_{u'}$ is a tree because $(H,\cV)$ is connected-$r$-acyclic for $r \geq 2$.
    Thus, the $H^*$-path $Q$ in $H_u$ yields that $H_{u'} \subseteq H^\ast$ avoids $h_i$ for some $i=1,2$.
    As $H_u \cup H_{u'}$ is a tree, there is a $Q$--$H_{u'}$ path $Q'$ in $H_u$.
    If $H_{u'}$ contains the other $h_j$ with $j=3-i$, then we choose $Q'$ as the trivial path $h_j$.
    The choice of $Q'$ guarantees that some $h_k$ with $k =1,2$, say $h_1$, is neither in $H_{u'}$ nor in $Q'$.
    
    Let $v_1 \in N^*$ with $h_1 \in H_{v_1}$ and fix a shortest $v$--$v_1$ path $P_1$ in $G$. 
    Fix a shortest $v$--$u'$ path $P$ in $G$.
    Since $u' \in N'$, it follows that $|P| \leq d'+1 = d$.
    Similarly, since $v_1 \in N^*$, it follows that $|P_1| \leq d+1$ if $r$ is odd and $|P_1| \leq d$ if r is even.
     Since $uu'$ is an edge of $G$ and $V(P_1) \cup V(P)$ is connected in $G$, the set $U := V(P_1) \cup V(P) \cup \{u\}$ is also connected in $G$ and has size at most 
    $$|V(P_1) \cup V(P) \cup \{u \}| \leq |V(P_1)| + |V(P)| - |\{v\}| + |\{u\}| \leq r.$$ 
    As $U \setminus \{u\} = V(P_1) \cup V(P)$ is connected in $G$, \cref{lem:strongerH2} yields that $H_{U\setminus \{u\}}$ is connected.
    Let $Q^*$ be the path in $Q \cup Q'$ that starts in $h_1$ follows $Q$ until it meets $Q'$ and then follows $Q'$ until it meets $H_{U \setminus \{u\}}$.
    Note that this will happen at latest at the end of $Q'$ in $H_{u'}$, which is in $H_{U \setminus \{u\}}$.
    The choice of $Q'$ guarantees that $Q^*$ is non-trivial, as $Q'$ does not contain $h_1$. 
    Hence, $Q^*$ is an $H_{U \setminus \{u\}}$-path in $H_u$, so $Q^*$ closes to a cycle in $H_U$, which contradicts that $(H,\cV)$ is connected-$r$-acyclic.
\end{proof}

\begin{proof}[Proof of \cref{main:InducingTDsonBalls}]
Let $G$ be a graph and let $(H, \mathcal{V})$ be a graph-decomposition of $G$ into cliques. By \cref{lem:ConAcyclicEquivToAcyclicForIntoCliques}, $(H, \mathcal{V})$ is connected-$r$-acyclic if and only if it is $r$-acyclic. Furthermore, by \cref{thm:connectedAcyclicEquivLocallyInducingTDs}, $(H, \mathcal{V})$ $r$-locally induces tree-decompositions if and only if it is $r$-acyclic. Now, \cref{main:InducingTDsonBalls} follows from \cref{mainresult:racycliccliquegraphs-iff-r-locally-chrodal-etc}.
\end{proof}

\subsection{Locally deriving \td s into cliques}
Since every $r/2$-ball is chordal in an $r$-locally chordal graph, it is natural to hope that the tree-decompositions $r$-locally induced by $r$-acyclic graph-decompositions into cliques are themselves tree-decompositions into cliques, especially given the characterization in \cref{lemma:CharacterisingLocallyInducingTDs}. If $r$ is odd, then the $r$-locally induced decompositions are indeed into cliques if and only if $(H,\cV)$ is into cliques, since the $(r/2)$-balls are then induced subgraphs.
Unfortunately, if $r$ is even, even though the decompositions $r$-locally induced by an $r$-acyclic graph-decomposition $(H,\cV)$ into cliques are tree-decompositions, they might not be into cliques.
But we can derive from $(H,\cV)$ tree-decompositions of its $r/2$-balls into cliques, as follows.

First, let us observe the following corollary of \cref{lem:edge-on-pre-last-level}:

\begin{corollary}\label{cor:OuterVerticesAreInUniqueMaximalCliques}
    Let $v$ be a vertex in an (possibly infinite) $r$-locally chordal graph $G$, and let $r \geq 2$ be an even integer.
    Then, for every vertex $u \in N^{r/2}(v)$, its neighborhood $N(u) \cap N^{r/2-1}(v)$ in $B_{r/2}(v)$ forms a clique. \qed
\end{corollary}

Let $G$ be an $r$-locally chordal graph and let $(H,\cV)$ be an $r$-acyclic decomposition of $G$ for $r \geq 3$.
Assume now that $r$ is even.
Fix a vertex $v$ of $G$.
Obtain $H^{v,r,*}$ from the disjoint union of $H^{v,r-1}$ and one single node $h_u$ for each $u \in N^{r/2}(v)$ by adding a single edge $f_u = h_u h_u'$ where $h_u'$ is a node of $H^{v,r-1}$ whose bag $V^{v,r-1}_{h_u'}$ contains the neighborhood $N(u) \cap N^{r/2-1}(v)$, which exists by \cref{lem:bag-cont-clique-in-3-acyclic}, as $N(u) \cap N^{r/2-1}(v)$ is a clique by \cref{cor:OuterVerticesAreInUniqueMaximalCliques}.
Consider a node $h$ of $H^{v,r,*}$.
If $h$ is a node of $H^{v,r-1}$, then $V^{v,r,*}_h := V^{v,r-1}_h$.
Otherwise, there exists a unique $u \in N^{r/2}(v)$ such that $h = h_u$.
Then $V^{v,r,*}_h := \{u\} \cup (N(u) \cap N^{r/2-1}(v))$.
We refer to the decomposition $(H^{v,r,*}, \cV^{v,r,*})$ as a decomposition \defn{$r$-locally derived by $(H,\cV)$}.
If $r$ is odd, then we refer to the decomposition $r$-locally induced by $(H,\cV)$ to a decomposition \defn{$r$-locally derived by $(H,\cV)$}.

\begin{theorem}
\label{thm:locally-derived-TDs}
    Let $(H,\cV)$ be an $r$-acyclic graph-decomposition of a (possibly infinite) graph $G$ into cliques with $r \geq 3$.
    Then the decompositions $r$-locally derived by $(H,\cV)$ are tree-decompositions into cliques.
\end{theorem}

\begin{proof}
    The decomposition $(H,\cV)$ witnesses that $G$ is $r$-locally chordal by \cref{mainresult:racycliccliquegraphs-iff-r-locally-chrodal-etc}.
    Fix a vertex $v$ of $G$.
    If $r$ is odd, then this follows from \cref{main:InducingTDsonBalls} and the definition of the decompositions $r$-locally induced by $(H,\cV)$, as $B_{r/2}(v)$ is an induced subgraph of $G$.
    Thus, assume that $r$ is even.
    Since $r-1$ is odd, the decomposition $(H^{v,r-1}, \cV^{v,r-1})$ $(r-1)$-locally induced by $(H,\cV)$ is a tree-decomposition of $B_{(r-1)/2}(v)$ into cliques.
    It is immediate from the above construction of the decomposition $(H^{v,r,*}, \cV^{v,r,*})$ $r$-locally derived at $v$ that it is a tree-decomposition of $B_{r/2}(v)$ into cliques.
\end{proof}

\section{Characterization via graph-decompositions that 
locally induce local separations}\label{sec:GDofLocChord}

A well-known key property of a tree-decomposition $(T,\cV)$ of a graph $G$ is that every edge $f$ of the decomposition tree $T$ \defn{induces a separation $s^f$ of $G$} whose separator is the adhesion set $V_f$ corresponding to $f$, see e.g.\ \cite[Lemma~12.3.1]{bibel}. Formally:

\begin{enumerate}[label=(T\arabic*$^\prime$)]
    \setcounter{enumi}{1}
    \item\label{T2'} For every edge $f = t_1t_2$ of $T$ and for $i = 1,2$, set $A^f_i \coloneqq \bigcup_{t \in V(T_i)} V_t$, where $T_i$ is the component of $T-f$ which contains~$t_i$. Then $\mathdefn{s^f} \coloneqq\{A^f_1,A^f_2\}$ is a separation of $G$ and its separator $A^f_1 \cap A^f_2$ is the adhesion set $V_f = V_{t_1} \cap V_{t_2}$ corresponding to $f$.
\end{enumerate}
Indeed, it is well-known that a pair $(T,\cV)$ is a tree-decomposition if and only if it satisfies \cref{axiomH1} and \cref{T2'}.

A \defn{bond} of a graph is an inclusion-minimal cut of the graph.
Note that the bonds of a tree $T$ are precisely the edge-subsets of $T$ consisting of single edges.
Analogous to \cref{T2'}, one may prove along the lines of the proof of \cite[Lemma~3.4]{canonicalGD} that every bond $F$ of the model graph $H$ of a \gd\ $(H,\cV)$ of $G$ induces a separation $s^F$ of $G$:

\begin{enumerate}[label=(H\arabic*$^\prime$)]
    \setcounter{enumi}{1}
    \item\label{H2'} For every bond $F = E_G(A_1,A_2)$ of $H$ and for $i = 1,2$, set $A^F_i \coloneqq \bigcup_{h \in A_i} V_h$. 
    Then $\mathdefn{s^F} \coloneqq \{A^F_1,A^F_2\}$ is a separation of $G$ and its separator $A^F_1 \cap A^F_2$ is precisely the union $\bigcup_{f \in F} V_f$ of the adhesion sets $V_f = V_{h_1} \cap V_{h_2}$ corresponding to edges $f = h_1h_2$ in the bond $F$.
\end{enumerate}

In view of \cref{H2'}, a graph-decomposition represents the structure of the underlying graph.
Indeed, a pair $(H,\cV)$ is a \gd\ if and only if it satisfies \cref{axiomH1} and \cref{H2'} (for more details see \cite{GDAxiomsAbrishami}).

Given a \gd\ $(H,\cV)$ of a graph $G$, the edges of $H$ do not satisfy a property similar to \cref{T2'} in general; if the model graph is not a tree, there is no reason to hope that single edges of $H$ correspond to separations of $G$. 
However, they do sometimes correspond to a local analog of separations. A recent paper \cite{computelocalSeps} investigates the relationship between graph-decompositions and local separations. 
It turns out that, since the graph-decompositions in our context are very well-structured, we can also prove that they interact nicely with local separations. The goal of this section is to make this interaction precise. We will at the end of the section be able to prove in \cref{mainthm:gd} that $r$-locally chordal graphs are precisely those graphs which admit graph-decompositions whose edges induce $r$-local separations in an obvious way.

\subsection{Additional properties of graph-decompositions.} 
Let $(H, \cV)$ be a graph-decomposition of $G$.
Note that, for a vertex-subset $X$ of $G$, the graph $H_X = \bigcup_{x \in X} H_x$ is the subgraph of $H$ formed by those nodes $h$ and those edges $f=h_0h_1$ whose corresponding bags $V_h$ and adhesion sets $V_{f} = V_{h_0} \cap V_{h_1}$ meet $X$.
In the following, we consider a \gd\ $(H,\cV)$ which additionally satisfy

\begin{enumerate}[label=(H\arabic*T)]
    \setcounter{enumi}{1}
    \item \label{H2T} $H_v$ is a tree for every vertex $v$ of $G$, and
\end{enumerate}
\begin{enumerate}[label=(H\arabic*)]
    \setcounter{enumi}{2}
    \item \label{H3} $H_u \cap H_v = H[\{h \in V(H) \mid u,v \in V_h\}]$ is connected for every edge $uv$ of $G$.
\end{enumerate}

We remark that both \cref{H2T} and \cref{H3} hold for every \td . \cref{H3} also holds for a \gd\ obtained from a \td\ of $G_r$ by folding via $p_r$ \cite[Paragraph after Lemma~3.9]{canonicalGD}.
It turns out that the properties \cref{H2T} and \cref{H3} are equivalent to being connected-2-acyclic. 

\begin{lemma}\label{lem:2acyclic-equiv-H2T-H3}
    A \gd\ $(H,\cV)$ of a (possibly infinite) graph $G$ satisfies \cref{H2T} and \cref{H3} if and only if it is connected-$2$-acyclic.
\end{lemma}

\begin{proof}
    \cref{H2T} says that the regions $H_v$ are trees. Therefore, the equivalence follows from \cref{lem:glueingtrees-to-trees}.
\end{proof}

Note that, for an honest graph-decomposition $(H,\cV)$ satisfying \cref{H2T} and \cref{H3}, the only valid choice of co-parts $H_v$ is the subgraphs $H[W_v]$ induced by the co-bags $W_v$ by \cref{lemma:HvAreInduced}.

\subsection{Interplay: Graph-decompositions \texorpdfstring{\&}{and} local separations}

We now formally introduce what a local separation from \cite{computelocalSeps} is. 
For a vertex-subset $X$ of $G$, we denote the set of edges of $G$ with exactly one end in $X$ by \defn{$\partial_G X$}.
A \defn{pre-separation} of $G$ is a set $\{E_0,X,E_1\}$ with $X \subseteq V(G)$ and disjoint $E_0, E_1$ whose union is $\partial_G X$.
We remark that every separation $\{A,B\}$ of $G$ \defn{induces the pre-separation} $\mathdefn{\pre(\{A,B\})} \coloneqq \{\partial_{G} (A\setminus B), A \cap B, \partial_{G} (B \setminus A)\}$, see \cref{fig:sepIndPreSep}.

\begin{figure}[ht]
    \centering
    \begin{subfigure}[b]{0.49\textwidth}
        \centering
        \includegraphics[width=0.4\textwidth]{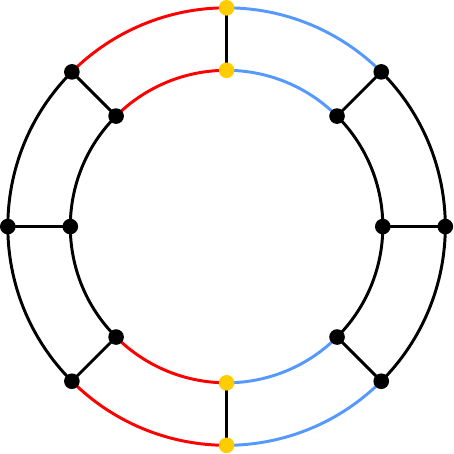}
        \caption{A pre-separation induced by a separation.}
        \label{fig:sepIndPreSep}
    \end{subfigure}
    \hfill
    \begin{subfigure}[b]{0.49\textwidth}
        \centering
        \includegraphics[width=0.4\textwidth]{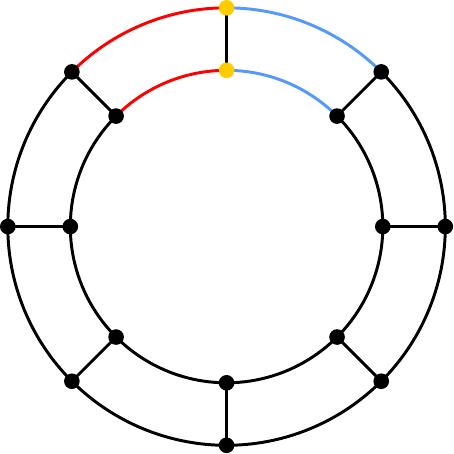}
        \caption{A local separation not induced by a separation.}
        \label{fig:LocSep}
    \end{subfigure}
    \caption{Examples of local separations $\{E_0,X,E_1\}$. The partition classes of $\{E_0,E_1\}$ are depicted in red and blue, respectively. The set $X$ is formed by the yellow vertices.}
    \label{fig:SepLocSep}
\end{figure}

Let $X$ be a vertex-subset of $G$.
A walk which starts and ends in $X$ but internally avoids $X$ is an \defn{$X$-walk}. 
An \defn{$r$-local $X$-walk} in $G$ is an $X$-walk that is contained in a cycle of length at most $r$ in $G$ or consists of a single edge with both ends in $X$.
We remark that every reduced non-trivial $r$-local $X$-walk is either an $X$-path or a walk once around a cycle of length at most $r$ that meets $X$ precisely once.
A pre-separation $\{E_0,X,E_1\}$ of $G$ is an \defn{$r$-local separation} of $G$ if there is no $r$-local $X$-walk which traverses an edge of $E_0$ and an edge of $E_1$.\footnote{It is easy to see that this definition is equivalent to the one given in \cite{computelocalSeps} but avoids the need of introducing \emph{$r$-local components}.} See \cref{fig:SepLocSep} for examples.
The pre-separations $\pre(s)$ induced by separations $s$ of $G$ are $r$-local separations. We also recall the following useful fact about $r$-local separations. 

\begin{lemma}
    Let $G$ be a (possibly infinite) graph, $r$ a positive integer, and $\{E_0, X, E_1\}$ an $r$-local separation of $G$. Then, a cycle of length at most $r$ meets each of $E_0$ and $E_1$ an even number of times. 
\end{lemma}
\begin{proof}
    Let $O$ be a cycle in $G$ of length at most $r$.
    By symmetry, it suffices to show that $O$ meets $E_1$ an even number of times.
    If $O$ does not intersect $X$, then $O$ does not meet $E_1$, as desired.
    Otherwise, $O$ meets $X$.
    Then, the reduced non-trivial $r$-local $X$-walks in $O$ partition the edge set of $O$.
    As each of them meets $E_1$ exactly zero or two times, $O$ meets $E_1$ an even number of times.
\end{proof}

Graph-decompositions are related to pre-separations in the following way. 

\begin{lemma}\label{lem:edge-induc-presep}
    Let $(H,\cV)$ be \gd\ of a (possibly infinite) graph $G$ which satisfies \cref{H2T} and \cref{H3}.
    Fix an edge $f = h_0h_1$ of $H$.
    For every vertex $v \in V_f$, let $H^v_0, H^v_1$ be the components of $H_v-f$ containing $h_0$ and $h_1$, respectively, and set $E^v_i \coloneqq \{vw \in \partial_G V_f \mid \exists h \in V(H^v_i) \colon w \in V_h \}$ for $i =0,1$.
    Then $E_i \coloneqq \bigcup_{v \in V_f} E_i^v$ forms a pre-separation $\{E_0,V_f,E_1\}$ of $G$.
\end{lemma}

\noindent 
We think that it may be possible that satisfying \cref{H2T} and \cref{H3} is not the only way to guarantee that every edge of the model graph induces a pre-separation which has the potential to be a local separation.
However, in this paper, our constructed \gd s always satisfy \cref{H2T} and \cref{H3}.
Thus, 
in the context of \cref{lem:edge-induc-presep}, the edge $f$ \defn{induces the pre-separation} $\{E_0,V_f,E_1\}$ of $G$.
If $\{E_0,V_f,E_1\}$ is an $r$-local separation of $G$, then $f$ \defn{induces an $r$-local separation}.

\begin{proof}    
    Fix $v \in V_f$. 
    By \cref{axiomH1}, every edge $vw$ in $\partial_G V_f$ is contained in the union of the two sets $E^v_0$ and $E^v_1$.
    It remains to show that $E^v_0$ and $E^v_1$ are disjoint.
    Since $H_v$ is a tree by \cref{H2T}, the two components $H_0^v$ and $H_1^v$ are distinct.
    Consider an edge $vw \in \partial_G V_f$. 
    Since $vw \in \partial_G V_f$, it follows that $w \not \in V_f$, and thus that $f$ is not in the tree $H_w$. Since $(H, \cV)$ satisfies \cref{H3}, we know that $H_w \cap H_v$ is connected, so $H_w$ meets exactly one of $H_1^v$ and $H_0^v$. As edge $vw$ was chosen arbitrarily, this proves that $E_0$ and $E_1$ are disjoint, and thus that $\{E_0, V_f, E_1\}$ is a pre-separation of $G$.
\end{proof}

We are now ready to state the main result of this section. A \gd\ $(H, \cV)$ of $G$ \defn{induces $r$-local separations} if $(H, \cV)$ satisfies  \cref{H2T} and \cref{H3} and the pre-separation induced by each edge of $H$ is an $r$-local separation of $G$.

\begin{mainresult}
\label{mainthm:gd}
    Let $G$ be a locally finite graph and $r \geq 3$ an integer.
    The following are equivalent:
    \begin{enumerate}
        \item \label{gd:r-local} $G$ is $r$-locally chordal.
        \item \label{gd:maxcliques} $G$ admits a graph-decomposition into its maximal cliques which induces $r$-local separations.
        \item \label{gd:canonical} $G$ admits a canonical graph-decomposition into cliques which induces $r$-local separations.
        \item \label{gd:cliques} $G$ admits a graph-decomposition into cliques which induces $r$-local separations.
    \end{enumerate}
\end{mainresult}

We conclude this subsection with a more precise statement regarding the structure of pre-separations in graphs with \gd s into cliques which we will need later:

\begin{lemma}\label{lem-cycle-crossing-presep}
    Let $(H,\cV)$ be a \gd\ of a (possibly infinite) graph $G$ into cliques  which satisfies \cref{H2T} and \cref{H3}.
    Suppose $X \subseteq V(G)$ is inclusion-minimal such that $H_X$ contains a cycle.
    Then, $G[X]$ is an induced cycle, and there exists an edge $f$ in a cycle $C$ in $H_X$ with induced pre-separation $\{E_0,V_f,E_1\}$ such that $G[X]$ meets $V_f$ in a single vertex whose two incident edges in $G[X]$ are contained in distinct $E_i$.
\end{lemma}
\begin{proof}
Without loss of generality, we may assume that $H = H_{V(G)}$, i.e.\ the decomposition $(H,\cV)$ is honest, and thus the only valid choice for the co-parts are the induced subgraph $H[W_v]$ by \cref{lemma:HvAreInduced}.
By \cref{lem:2acyclic-equiv-H2T-H3}, $(H, \cV)$ is connected-2-acyclic.

Let $X \subseteq V(G)$ be inclusion-minimal such that $H_X$ contains a cycle $C$. 
By \cref{lem:converseofinducedcycle-yields-cycleHX} (via \cref{thm:rig-is-gd-into-cliques}), $G[X]$ is an induced cycle of $G$.
Fix $v \in X$. Since $X$ is inclusion-minimal with the property that $H_X$ contains $C$, there is an edge $f$ of $C$ such that $v$ is the unique vertex of $X$ whose co-part $H_v$ contains $f$, i.e.\ $f \in E(C \cap H_v) \setminus E(H_{X-v})$, where $X-v := X \setminus \{v\}$. Consider the pre-separation $\{E_0, V_f, E_1\}$ of $G$ induced by $f$. Since $f \not \in H_x$ for any $x \in X - v$, it follows that $V_f \cap X = \{v\}$, as the $H_x$ are the induced subgraphs $H[W_v]$.

Let $P$ be an inclusion-maximal connected subgraph of $(C \cap H_v) - E(H_{X-v})$ containing $f$.
Since $P \subseteq H_v$ and $(H, \cV)$ is connected-2-acyclic, it follows that $P \neq C$, so $P$ is a subpath of $C$.
Let $h_0$ and $h_1$ be the ends of $P$.
By the choice of $P$, there are $x_0,x_1 \neq v$ in $X$ such that $h_i$ is a node of $H_{x_i}$ for $i = 0,1$.
Since $(H,\cV)$ is $2$-acyclic, $x_0 \neq x_1$.
Now, up to symmetry, $vx_0 \in E_0$ and $vx_1 \in E_1$. This completes the proof. 
\end{proof}

\subsection{Equivalence: \texorpdfstring{$r$}{r}-acyclic \& inducing local separations}
Our main tool to prove \cref{mainthm:gd} is an equivalence between being $r$-acyclic and inducing $r$-local separations for \gd s into cliques. We begin by showing that every graph-decomposition which is connected-$r$-acyclic induces $r$-local separations. 

\begin{theorem}\label{thm:connected-acyclic-IMPLIES-separations}
    Let $G$ be a (possibly infinite) graph and $r \geq 3$ an integer. Suppose $(H, \mathcal{V})$ is a connected-$r$-acyclic graph-decomposition of $G$. Then $(H, \cV)$ induces $r$-local separations. 
\end{theorem}
\begin{proof}
    Since $r \geq 2$, not only is $H_v$ a tree for every vertex $v$ of $G$, i.e.\ $(H,\cV)$ satisfies \cref{H2T}, but also $H_u \cup H_v$ is a tree for every edge $uv$ of $G$, so $(H,\cG)$ satisfies \cref{H3} by \cref{lem:glueingtrees-to-trees}.
    Let $f = h_1h_2$ be an edge of $H$, and let $\{E_1,V_f,E_2\}$ be the pre-separation induced by $f$ from \cref{lem:edge-induc-presep}.
    
    Consider an $r$-local $V_f$-walk $W$ in $G$ which contains at least one internal vertex.
    Let $v_1,v_2$ be its ends, and let $U$ be the set of vertices visited by $W$.
    In particular, $U$ is a connected set of at most $r$ vertices.
    Thus, since $(H,\cV)$ is $r$-connected-acyclic, $H_U$ is a tree.
    Let $H_i$ be the component of $H_U-f$ containing $h_i$ for $i = 1,2$.
    
    As $U' := U \setminus \{v_1,v_2\}$ is connected in $G$, \cref{lem:strongerH2} yields that $H_{U'}$ is connected.
    Since every $u \in U'$ is not in $V_f$, the edge $f$ is not in its corresponding co-part $H_u$.
    Hence, $H_{U'}$ meets precisely one of $H_1$ and $H_2$.
    In particular, all $H_u$ with $u \in U'$ meet the same $H_i$.
    Thus, as $H_i^{v_j}$ (defined as in \cref{lem:edge-induc-presep}) is a subgraph of $H_i$ for every $j = 1,2$ by definition, the first and last edge of $W$ are in the same $E_i$.
    It follows that $\{E_1,V_f,E_2\}$ is an $r$-local separation of $G$, as desired.
\end{proof}

Next, we show that the converse also holds when $(H, \cV)$ is into cliques. 

\begin{theorem}\label{thm:rlocalGD-equiv-racyclicRIG}
    Let $G$ be a (possibly infinite) graph and $r \geq 3$ an integer.
    A graph-decomposition $(H,\cV)$ of $G$ into cliques is $r$-acyclic if and only if it induces $r$-local separations.
\end{theorem}
\begin{proof}
    Let $(H,\cV)$ be a \gd\ of $G$ into cliques.
    Assume that $(H,\cV)$ is $r$-acyclic.
    By \cref{thm:connected-acyclic-IMPLIES-separations}, $(H, \cV)$ induces $r$-local separations. 
    
    Conversely, assume that $(H,\cV)$ induces $r$-local separations. 
    Suppose for a contradiction that $(H,\cV)$ is not $r$-acyclic.
    Then there is a set $X$ of at most $r$ vertices such that $H_X = \bigcup_{x \in X} H_x$ contains a cycle.
    We may choose $X$ to be inclusion-minimal.
    By \cref{lem-cycle-crossing-presep}, there is an edge $f$ in a cycle $O$ of $H_X$ with its induced pre-separation $\{E_0,V_f,E_1\}$ such that the induced cycle $C \coloneqq G[X]$ meets $V_f$ in a single vertex $y$ and its two incident edges $e_0,e_1$ on $C$ satisfy $e_i \in E_i$ for $i =0,1$.
    Since $X$ meets $V_f$ only in $y$ and $|X| \leq r$, a closed walk $W$ once around $C$ based at $y$ is an $r$-local $V_f$-walk.
    Then the two edges $e_0,e_1$ incident to $y$ on $C$ are contained in the same $E_i$ of the $r$-local separation $\{E_0,V_f,E_1\}$, which is a contradiction. 
\end{proof}

\subsection{Proof of \texorpdfstring{\cref{mainthm:gd}}{Theorem XX}}

We can now prove the main result of this section. 

\begin{proof}[Proof of \cref{mainthm:gd}]

    By \cref{mainresult:racycliccliquegraphs-iff-r-locally-chrodal-etc}, $G$ is $r$-locally chordal if and only if $G$ admits an $r$-acyclic \gd\ $(H,\cV)$ into cliques. By \cref{thm:rlocalGD-equiv-racyclicRIG}, $(H, \cV)$ induces $r$-local separations. By the definition, $(H, \cV)$ is canonical if and only if $v \mapsto H_v$ is canonical. By \cref{mainresult:racycliccliquegraphs-iff-r-locally-chrodal-etc}, we may additionally assume that $(H,\cV)$ is either canonical or into maximal cliques. This completes the proof. 
\end{proof}

\section{Computing clique graphs}\label{sec:ComputingCliqueGraphs}

In \cref{mainresult:racycliccliquegraphs-iff-r-locally-chrodal-etc}, we prove that a locally finite graph $G$ is $r$-locally chordal if and only if $G$ has an $r$-acyclic clique graph. In this section, we provide a simple and efficient algorithm to compute $r$-acyclic clique graphs of finite $r$-locally chordal graphs. This algorithm is a generalization of a well-known algorithm to compute clique trees of chordal graphs and a parallelization of Gavril's  construction to characterize wheel-free graphs \cite{gavrilAlgorithm}. 

The algorithm in particular finds a spanning subgraph of the intersection graph $\bK(G)$ of the maximal cliques of $G$. For the rest of the paper, we equip the edges of $\bK(G)$ with weights according to the number of vertices in the intersection they represent; i.e.\ the \defn{weight} of an edge $K_1K_2$ is $|K_1 \cap K_2|$.
Let $\bK$ be an edge-weighted graph.
The \defn{total weight} \defn{$w_{\bK}(H)$} of a subgraph $H$ of $\bK$ is the sum of the weights of the edges in $H$.
A \defn{maximum weight spanning tree of $\bK$} is a spanning tree of $\bK$ which has maximum total weight among all spanning trees of $\bK$. 

Maximum weight spanning trees of clique graphs are related to chordal graphs through the following theorem.

\begin{theorem}[\cite{gavril1987maximumweightspanningtree}, Theorem 1]\label{thm:maxweighttree-is-cliquetree}
    Let $G$ be a finite connected chordal graph.
    Clique trees of $G$ are precisely the maximum weight spanning trees of $\bK(G)$. 
\end{theorem}

\cref{thm:maxweighttree-is-cliquetree} is especially useful since there is an efficient algorithm to compute maximum weight spanning trees:

\begin{theorem}[Kruskal \cite{kruskal}]\label{thm:KRUSKALalgorithm}
    There is an algorithm, \cref{Kruskalalgorithm}, to compute a maximum weight spanning tree of every given finite connected edge-weighted graph $G$ on $n$ vertices with $m$ edges that runs in time $O(m \log n)$.
\end{theorem}

\begin{algorithm}[Kruskal \cite{kruskal}]\label{Kruskalalgorithm}
    Let $H$ be a finite connected edge-weighted input graph.
    \begin{enumerate} [label={(\arabic*)}]
        \item Order the edges of $H$ in order $e_1, \dots, e_m$ of decreasing weight.
        \item Let $T_0$ be the edgeless graph with vertex set $V(H)$.
        \item For $i$ from $1$ to $m$, let $T_i \coloneqq T_{i-1} + e_i$ if $T_i +e_i$ is acyclic. Else, let $T_i \coloneqq T_{i-1}$
        \item After step $m$, set $T \coloneqq  T_m$.
    \end{enumerate}
\end{algorithm}

In view of the previous three results, computing a clique tree of a chordal graph is equivalent to computing a maximum weight spanning tree of $\bK(G)$. In order to assess the running time of such an algorithm, we need to understand the size of $\bK(G)$.

\begin{proposition}[Folklore, e.g.\ \cite{gavrilAlgorithm,Rose}]\label{prop:maximumnumberofmaxcliques-in-chordalgraph}
    A  chordal graph $G$ on $n$ vertices has at most $n$ maximal cliques. Thus, $\bK(G)$ has at most $\binom{n}{2} \leq n^2$ edges.  
\end{proposition}

All in all, one derives the following from \cref{thm:maxweighttree-is-cliquetree}, \ref{thm:KRUSKALalgorithm} and \cref{prop:maximumnumberofmaxcliques-in-chordalgraph}.

\begin{theorem}\label{ALGOforCliqueTree}
    Given a finite connected chordal graph $G$ on $n$ vertices, computing $\bK(G)$ and running \cref{Kruskalalgorithm} on $\bK(G)$ yields a clique tree of $G$ in running time $O(n^2 \log n)$. 
\end{theorem}

In this section, we provide a parallel algorithm that, when $G$ is a finite $r$-locally chordal graph, computes an $r$-acyclic clique graph of $G$.

\begin{mainresult}\label{thm:correctalgo}
    There is a parallel algorithm, \cref{THEalgorithm}, to compute an $r$-acyclic clique graph of every given finite $r$-locally chordal graph $G$ that runs on $n = |V(G)|$ processors in time $O(\Delta^2 \log \Delta)$, where $\Delta$ is the maximum degree of $G$.
\end{mainresult}

\noindent We remark that Gavril's proof \cite[Theorem~5]{nbrhood-chordal} of \cref{mainthm:wheelfree}~(i)$\to$(ii) is an algorithm with running time $O(n^3 \cdot \log n)$.\footnote{In fact, Gavril claims that his algorithm has running time $O(n^3)$. As far as we understand his argument, it only yields $O(n^3 \cdot \log n)$.}

By \cref{lem:char-locallychordal-via-3-acyclic-region-rep}, it suffices to consider the case $r= 3$.
Our algorithm is essentially a parallel version of Gavril's.
It works as follows.

\begin{mainalgorithm}\label{THEalgorithm}
Let $G$ be a $3$-locally chordal (equivalently: wheel-free) input graph on $[n]$.
Run on processor $v \in [n]$ the following:
\begin{enumerate}[label={(\arabic*)}]
    \item Compute the intersection graph $\bK(B_{3/2}(v))$ of the maximal cliques of $B_{3/2}(v)$. 
    \item Order the edges of $\bK(B_{3/2}(v))$ in order $e^v_1, \dots, e^v_{m(v)}$ of decreasing weight; decide ties of two edges by the lexicographic order on the union of their respective ends. 
    \item Let $H^v_0$ be the edgeless graph with vertex set $V(\bK(B_{3/2}(v))$. 
    \item
    For $i$ from $1$ to $m(v)$, let $H^v_i \coloneqq H^v_{i-1} + e^v_i$ if $H^v_{i-1} + e^v_i$ is acyclic. Else, let $H^v_i \coloneqq H^v_{i-1}$.
    \item After step $m(v)$, set $H^v \coloneqq H^v_{m(v)}$.
\end{enumerate}
After all $n$ processor are finished, set $H \coloneqq \bigcup_{v \in V(G)} H^v$.
\end{mainalgorithm}

\subsection{Analysis of the algorithm}
\cref{THEalgorithm} essentially works by running Kruskal's algorithm to find a clique tree of $B_{3/2}(v)$ for every $v \in V(G)$ simultaneously. In this section we prove that this algorithm indeed returns an $r$-acyclic clique graph when the input graph is $r$-locally chordal. 

For this, by \cref{lem:char-locallychordal-via-3-acyclic-region-rep}, we only need to ensure that the spanning subgraph $H$ of $\bK(G)$ returned by \cref{THEalgorithm} is a $3$-acyclic clique graph.
Indeed, it suffices to show that $H$ is a $2$-acyclic clique graph due to the following lemma.

\begin{lemma}\label{lem:cliquegraph-2acyclic-iff-3acyclic}
    Let $G$ be a (possibly infinite) graph.
    Then every $2$-acyclic clique graph $H$ of $G$ is Helly, and thus $3$-acyclic.
\end{lemma}

\begin{proof}
    Let $H$ be a $2$-acyclic clique graph of $G$, and let $v \mapsto H_v \coloneqq H[K_G(v)]$ be its corresponding region representation of $G$.

    By \cref{lem:Helly2-acyclic-equiv-3acyclic}, the $2$-acyclic clique graph $H$ is $3$-acyclic if $H$ is Helly.
    To show that $H$ is Helly, let
    $X$ be a  finite vertex-subset of $G$ such that each two $H_x, H_y$ with $x,y \in X$ intersect.
    In particular, $X$ is a clique of $G$.
    Let $K$ be a maximal clique of $G$ containing $X$, which is thus a node of every $H_x$ with $x \in X$.
    Hence, \cref{lem:gavril-lemmas}~(i) yields that $H_X$ is a tree.
    Therefore, $x \mapsto H_x$ is a subtree representation of the complete graph $G[X]$ over the tree $H_X$.
    By \cref{thm:subtree-reps-are-helly}, $\bigcap_{x \in X} H_x$ is nonempty.
    Therefore, the $2$-acyclic clique graph $H$ is Helly, and thus $3$-acyclic.
\end{proof}

We note that we cannot push this further to $1$-acyclic due to the following example.

\begin{example}
    Consider the wheel $W_n$ with $n \geq 4$ whose hub is $v$ and whose rim is the cycle $v_1v_2 \dots v_nv_1$.
    Let $G$ be the graph obtained from $W_n$ by adding a vertex $x$ and the edges $xv_i$ for $i = 1,2$.
    Now, $G$ has a $1$-acyclic clique graph but not a $2$-acyclic clique graph. However, its induced subgraph $W_n$ does not even have a $1$-acyclic clique graph.
\end{example}

\begin{proof}
    \begin{claim}\label{claim:WheelNo1AcyclicCG}
        The wheel $W_n$ with $n \geq 4$ has no $1$-acyclic clique graph.
    \end{claim}
    \begin{claimproof}
        The maximal cliques of $W_n$ are $D_i \coloneqq \{v,v_i,v_{i+1}\}$ for $i \in [n]$, where $v_{n+1} = v_1$, so $\bK(W_n)$ is the cycle $D_1D_2 \dots D_nD_1$.
        Now consider a clique graph $H$ of $W_n$.
        Let $i \in [n]$. Consider $D_{0}$ to be $D_n$.
        Since $H_{v_i} = H[K_{W_n}(v_i)] = G[\{D_{i-1},D_i\}]$ is connected, the edge $D_{i-1}D_i$ is in $H$.
        Thus, $H$ is the whole cycle $\bK(W_n)$.
        Since $H_v = H[K_{W_n}(v)] = H$, the clique graph $H$ is not $1$-acyclic.
    \end{claimproof}

    \begin{claim}
        $G$ has a $1$-acyclic clique graph.
    \end{claim}

    \begin{claimproof}
        The maximal cliques of $G$ are $D_1,\dots, D_n$ and also $D \coloneqq \{x,v_1,v_2\}$.
        Hence, the clique graph $\bK(G)$ is obtained from $\bK(W_n)$ by adding $D$ and the edges $DD_{n}, DD_1, D_{2}$.
        Let $H$ be the spanning cycle $DD_1D_2 \dots D_nD$.
        We claim that $H$ is a $1$-acyclic clique graph of $G$.
        For $i \in [n] \setminus \{1,2\}$, $H_{v_i}$ consists of the single edge $D_{i-1}D_i$.
        $H_{v_1}$ is the path $D_nDD_1$, $H_{v_2}$ is the path $DD_1D_2$, and $H_x$ is the graph on the single vertex $D$.
        Moreover, $H_v$ is the path $D_1D_2 \dots D_n$.
        All in all, $G$ is a $1$-acyclic clique graph.
    \end{claimproof}

    To see that $G$ has no $2$-acyclic clique graph, one follows the argument presented in the proof of \cref{claim:WheelNo1AcyclicCG} to show that a $1$-acyclic clique graph $H$ of $G$ contains all but one edge of $\bK(W_n)$.
    Thus, the union of $H_v$ and either $H_{v_1}$ or $H_{v_2}$ forms a cycle, and thus $H$ is not $2$-acyclic. 
\end{proof}

The next lemma yields that the spanning subgraph $H$ of $\bK(G)$ returned by \cref{THEalgorithm} is $2$-acylic, if the parallel runs of the algorithm make the same decision about whether to contain an edge.

\begin{lemma}
\label{lem:clique-tree-sufficient}
    Let $G$ be a (possibly infinite) graph and 
    let $H$ be a spanning subgraph of $\bK(G)$. 
    Then $H$ is a $2$-acyclic clique graph of $G$ if and only if $H[K_G(v)]$ is a clique tree of $B_{3/2}(v)$ for all vertices $v$ of $G$. 
\end{lemma}
\begin{proof}
    Set $H_v \coloneqq H[K_G(v)]$ for every vertex $v$ of $G$.
    By \cref{propertiesofspanningsubgraph-of-CIG}, $H_u \cap H_v = H_v[K_{B_{3/2}(v)}(u)]$ for every two vertices $u,v$ of $G$. 
    
    Fix a vertex $v$ of $G$. The vertex set of $K_G(v)$ is exactly the set of maximal cliques containing $v$, which is also exactly the set of maximal cliques of $B_{3/2}(v)$. Therefore, $H[K_G(v)]$ is a clique tree of $B_{3/2}(v)$ if and only if $H_v$ is a tree and $ H_v[K_{B_{3/2}(v)}(u)] = H_u \cap H_v$ is connected for every $u \in B_{3/2}(v)$. Since this holds for every vertex $v$ of $G$, this is equivalent to requiring that $H_v$ is a tree for every vertex $v$ of $G$, i.e. $H$ satisfies \cref{H2T}, and $H_v \cap H_u$ is connected for every edge $uv$ of $G$, i.e. $H$ satisfies \cref{H3}. By \cref{lem:2acyclic-equiv-H2T-H3}, $H$ is connected-2-acyclic if and only if $H$ satisfies \cref{H2T} and \cref{H3}. Since $H_u$ and $H_v$ intersect in a clique graph of $G$ if and only if $uv$ is an edge of $G$, it follows that a clique graph is 2-acyclic if and only if it is connected-2-acyclic. This completes the proof. 
    \end{proof}

Now, we combine the above observations to prove \cref{thm:correctalgo}.

\begin{proof}[Proof of \cref{thm:correctalgo}]
    
    We claim that the spanning subgraph $H$ outputted by \cref{THEalgorithm} is indeed an $r$-acyclic clique graph of $G$. Afterwards we discuss its running time.
    
    \vspace{0.5em}
    
    \noindent{\bf Correctness.}
    As $K_G(v) = K(B_{3/2}(v))$ by \cref{propertiesofspanningsubgraph-of-CIG} and thus $\bK(G)[K_G(v)] = \bK(B_{3/2}(v))$, the graph $H = \bigcup_{v \in V(G)} H^v$ is a spanning subgraph of $\bK(G) = \bigcup_{v \in V(G)} \bK(G)[K_G(v)] = \bigcup_{v \in V(G)} \bK(B_{3/2}(v))$.
    Note that $H^v$ is a spanning tree of $H_v \coloneqq H[K_G(v)]$, as $H^v$ is obtained by running Kruskal's algorithm to obtain a maximum weight spanning tree of $\bK(B_{3/2}(v))$.
    In particular, $H_v$ is connected.
    Therefore, $H$ is a clique graph.

    It remains to show that $H$ is $r$-acyclic.
    As $G$ is $r$-locally chordal, it suffices to show that $H$ is $3$-acyclic by \cref{lem:char-locallychordal-via-3-acyclic-region-rep}.
    As, by \cref{lem:cliquegraph-2acyclic-iff-3acyclic}, a clique graph is $2$-acyclic if and only if it is $3$-acyclic, \cref{lem:clique-tree-sufficient} ensures that it suffices that $H_v = H[K_G(v)]$ is a clique tree of $B_{3/2}(v)$.
    As $H^v$ is obtained by running Kruskal's algorithm on $\bK(B_{3/2}(v))$, $H^v$ is a clique tree of $B_{3/2}(v)$ by \cref{ALGOforCliqueTree}.
    Thus, we need to show that the subgraph $H_v$ is indeed $H^v$.
    This follows from the following.
    
    \begin{claim}
        For every two vertices $u,v$ of $G$, the graphs $H^u$ and $H^v$ agree on $K_G(u) \cap K_G(v)$.
    \end{claim}

    \begin{claimproof}
        Order the edges of $\bK(G)$ in order $e_1,\dots,e_m$ of decreasing weight where we decide ties of two edges by the lexicographic order on the union of their respective ends.
        Note that, for every vertex $v$ of $G$, the order $e_1^v, \dots, e_{m(v)}^v$ is the order of the edges of $\bK(B_{3/2}(v))= \bK(G)[K_G(v)]$ induced by $e_1, \dots, e_m$.
        For a vertex $v$ of $G$ and an index $i = 1, \dots, m$, let $I(i,v)$ be the maximum $j = 1, \dots, m(v)$ such that $e_1^v, \dots, e_j^v$ is a subsequence of $e_1,\dots, e_i$. 

        Let $u$ and $v$ be two vertices of $G$.
        If $K_G(u) \cap K_G(v)$ is empty, then the claim holds.
        Hence, we may assume that $K_G(u)$ and $K_G(v)$ meet.
        
        We proceed by induction on $i = 0,1, \dots, m$ to show that $H^u_{I(i,u)}$ and $H^v_{I(i,v)}$ agree on $K_G(u) \cap K_G(v)$.
        The base case $i = 0$ holds by definition of $H^u_0$ and $H^v_0$.
        Suppose that $H^u_{I(i,u)}$ and $H^v_{I(i,v)}$ agree on $K_G(u) \cap K_G(v)$ for some $i = 0, \dots, m-1$.
        As we are otherwise done, we assume that $e_{i+1}$ is an edge of $\bK(G)[K_G(u) \cap K_G(v)]$, i.e.\ $e_{i+1} = K_1K_2$ with $u,v \in K_j$ for $j = 1,2$.
        By symmetry, it suffices to show that if $H^u_{I(i,u)} + e_{i+1}$ contains a cycle $C$, then $H^v_{I(i,v)} + e_{i+1}$ contains the cycle $C$, as well.
        Note that the $K_1$--$K_2$ path $P \coloneqq C -e_{i+1}$ is in $H^u_{I(i,u)}$.
        If $P \subseteq \bK(G)[K_G(u) \cap K_G(v)]$, then $P$ is in $H^v_{I(i,v)}$ by induction assumption, and thus $C \subseteq H^v_{I(i,v)} + e_{i+1}$.

        In the following we show that $P \subseteq \bK(G)[K_G(u) \cap K_G(v)]$.
        Since $H^u_{I(i,u)}$ is by definition a subgraph of the tree $H^u$, the unique $K_1$--$K_2$ path in the spanning tree $H^u$ of $\bK(G)[K_G(u)]$ is $P$.
        Recall that $H^u$ is even a clique tree of $B_{3/2}(u)$ by \cref{ALGOforCliqueTree}.
        Therefore, $$H^u_v \coloneqq H^u[K_{B_{3/2}(u)}(v)] = H^u[K_G(u) \cap K_G(v)] \subseteq \bK(G)[K_G(u) \cap K_G(v)]$$ is a subtree of $H^u$, which contains the nodes $K_1$ and $K_2$.
        In particular, the unique $K_1$--$K_2$ path in $H^u_v$ agrees with the unique $K_1$--$K_2$ path $P$ in $H^u$.
        So, $P \subseteq \bK(G)[K_G(u) \cap K_G(v)]$, as desired.
    \end{claimproof}

    \noindent{\bf Running time.}
    Let $v$ be a vertex of $G$.
    As $B_{3/2}(v)$ contains $d(v) +1$ many vertices and $d(v)$ is at most the maximum degree $\Delta$, \cref{ALGOforCliqueTree} yields that the algorithm on processor $v$ has running time $O(\Delta^2 \log \Delta)$.
\end{proof}

\subsection{Clique graphs as maximum weight spanning subgraphs}\label{subsec:cliqueGraphsMaxWeight}

Let $G$ be a graph.
A spanning subgraph $H$ of $\bK(G)$ is \defn{$r$-acyclic} if $(H_v \mid v \in V(G))$ for $H_v \coloneqq H[K_G(v)]$ is an $r$-acyclic family of subgraphs of $H$.
A \defn{maximum weight $r$-acyclic spanning subgraph of $\bK(G)$} is an $r$-acyclic spanning subgraph $H$ of $\bK(G)$ that has maximum total weight among all $r$-acyclic spanning subgraphs of $\bK(G)$. It turns out that, just as clique trees are precisely maximum weight spanning trees of $\bK(G)$ (\cref{thm:maxweighttree-is-cliquetree}), $r$-acyclic clique graphs are precisely the maximum weight $r$-acyclic spanning subgraphs of $\bK(G)$. Specifically, we prove:

\begin{mainresult}\label{thm:maximum-weight-acyclic-spanning}
    Let $G$ be a finite $r$-locally chordal graph for an integer $r \geq 3$. Let $H$ be a subgraph of $\bK(G)$.
    Then the following are equivalent:
    \begin{enumerate}
        \item \label{item:maxWeightAcyclicSpanning:1} $H$ is an $r$-acyclic clique graph of $G$.
        \item \label{item:maxWeightAcyclicSpanning:2} $H$ is a maximum weight $2$-acyclic spanning subgraph of $\bK(G)$.
        \item \label{item:maxWeightAcyclicSpanning:3} $H$ is a $2$-acyclic spanning subgraph of $\bK(G)$ of weight $W(G) \coloneqq (\sum_{K \in K(G)} |K|) - |V(G)|$.
    \end{enumerate}
\end{mainresult}

First, we need the following lemma.

\begin{lemma}\label{lem:magic-doublecounting}
    Let $G$ be a finite graph. Suppose that $H$ is a $1$-acyclic spanning subgraph of $\bK(G)$.
    Then the total weight of $H$ is at most $W(G) (\coloneqq (\sum_{K \in K(G)} |K|) - |V(G)|)$.
    Moreover, equality holds if and only if $H$ is a clique graph of $G$.
\end{lemma}

\begin{proof}
    Recall that an acyclic graph $T$ on $n$ vertices has at most $n-1$ edges.
    Using this, the desired inequality follows from double counting the total weight $w(H)$:
    \begin{align*}
        \sum_{K_1K_2 \in E(H)} |K_1 \cap K_2| & = \sum_{v \in V(G)} |E(H[K_G(v)])| & \\
        & \leq  \sum_{v \in V(G)} (|K_G(v)| - 1) & =  \left(\sum_{K \in K(G)} |K|\right) - |V(G)| = W(G).
    \end{align*}
    Recall that a spanning subgraph $H$ of $\bK(G)$ is a clique graph of $G$ if and only if all $H[K_G(v)]$ are connected
    Moreover, it is well-known that an acyclic graph $T$ on $n$ vertices is connected if and only if $T$ has $n-1$ edges (see e.g.\ \cite[Corollary~1.5.2]{bibel}).
    Thus, equality holds if and only if $H[K_G(v)]$ is connected for every vertex $v$ of $G$.
\end{proof}

\begin{proof}[Proof of \cref{thm:maximum-weight-acyclic-spanning}]
    First, consider an $r$-acyclic clique graph $H$ of $G$.
    By \cref{lem:clique-tree-sufficient}, the induced subgraph $H[K_G(v)]$ are in particular trees for every vertex $v$ of $G$.
    Thus, $H$ has total weight $W(G)$ by \cref{lem:magic-doublecounting}.

    Now, let $H$ be a maximum weight $2$-acyclic spanning $H$ subgraph of $\bK(G)$.
    As $G$ is $r$-locally chordal, $G$ has $r$-acyclic clique graphs by \cref{mainresult:racycliccliquegraphs-iff-r-locally-chrodal-etc}. 
    Thus, the above shows that the weight of $H$ is at least $W(G)$.
    As $H$ is in particular $1$-acyclic, \cref{lem:magic-doublecounting} yields that the weight of $H$ is at most $W(G)$, and thus precisely $W(G)$.
    Then the moreover-part of \cref{lem:magic-doublecounting} yields that $H$ is a clique graph.
    So, $H$ is a $2$-acylic clique graph, and thus $H$ is $3$-acyclic by \cref{lem:cliquegraph-2acyclic-iff-3acyclic}.
    Since $G$ is $r$-locally chordal, it follows from \cref{lem:char-locallychordal-via-3-acyclic-region-rep} that $H$ is even $r$-acyclic.

    All in all, this proves the equivalence of \cref{item:maxWeightAcyclicSpanning:1}, \cref{item:maxWeightAcyclicSpanning:2}, and \cref{item:maxWeightAcyclicSpanning:3}.
\end{proof}

\section{Future work}

\noindent In this section, we detail several directions for future work about locally chordal graphs. 

\subsection{Bounding the number of maximal cliques}\label{sec:NumberOfMaximalCliques}

One property of finite chordal graphs is that they have at most $n$ maximal cliques (\cref{prop:maximumnumberofmaxcliques-in-chordalgraph}), where $n$ is the number of their vertices. We can thus easily deduce that every $r$-locally chordal graph $G$ for $r \geq 3$ has at most $n^2$ maximal cliques by observing that $B_{3/2}(v)$ is chordal for every $v \in V(G)$ and that every maximal clique of $G$ lives in $B_{3/2}(v)$ for some $v$. However, we suspect that a better bound should hold for the number of maximal cliques of $r$-locally chordal graphs, where the bound depends on $r$ and approaches $n$ as $r$ tends to infinity. Specifically, we conjecture the following: 

\begin{conjecture}\label{conj:numberOfMaxCliques}
        For every $\eps > 0$, there exists an integer $r \geq 0$ such that the number of maximal cliques in finite $r$-locally chordal graphs $G$ is $O(|V(G)|^{1+\eps})$.
\end{conjecture}

An extremal case of $r$-locally chordal graphs are those whose $r/2$-balls are trees, i.e.\ precisely the graphs of girth $>r$.
Then \cref{conj:numberOfMaxCliques} holds, as the number of maximal cliques coincides with the number of edges in connected graphs $G$ of girth $>r$, which is at most $n^{1+\frac{2}{r-2}}$ if the graph $G$ has $n$ vertices and at least $2n$ edges \cite[\S III]{bollobasExtremalGT}.

\subsection{Algorithmic applications}\label{subsec:maxIndSet}
One benefit of structural decompositions like tree-decompositions is that they enable the design of algorithms for graphs with ``good'' decompositions. Indeed, tree-decompositions play a large role in algorithmic graph theory. In \cref{THEalgorithm}, we give an efficient algorithm to compute $r$-acyclic \gd s into the maximal cliques (equivalently: $r$-acyclic clique graphs) of $r$-locally chordal graphs, which are structural descriptions that generalize \td\ into the maximal cliques (equivalently: clique trees). Therefore, a very interesting follow-up question is whether these representations can be used in the design of efficient algorithms for locally chordal graphs. 

One of the most successful algorithmic applications of chordal graphs is to the \textsc{maximum independent set} (MIS) problem. The MIS problem asks, given a graph $G$ (and possibly a weight function $w: V(G) \to \mathbb{N}$ of the vertices), to return a maximum-size (or maximum-weight) independent set of $G$. MIS is a classic NP-hard problem in general \cite{karp}, and is even hard to approximate in general \cite{approximateMIS}. However, there is a linear-time algorithm to solve MIS in chordal graphs \cite{gavrilAlgorithm}. One can also solve MIS in chordal graphs using standard dynamic programming on a tree-decomposition into cliques. This does not require that the chordal graphs have bounded treewidth, because independent sets intersect cliques, and thus bags of tree-decompositions, in at most one vertex. The nice interplay between chordal graphs and the MIS problem has inspired several other approaches to and perspectives on the MIS problem, for example the {\em PMC method} \cites{lokshantov, fomin} and the width-parameter {\em tree-independence number} \cite{treealpha}. 

Since chordal graphs and MIS are such a fruitful pair, the MIS problem is a promising potential algorithmic application of locally chordal graphs. Our decompositions likewise have the property that they are into cliques, so a solution to MIS intersects each bag in at most one vertex. Since $r$-locally chordal graphs include, in particular, all graphs of girth greater than $r$, the class of graphs of high girth witnesses that MIS is NP-hard even in $r$-locally chordal graphs \cite{poljak}. However, we hope for an affirmative answer to the following. 
\begin{question}
    Is there a good approximation algorithm for \textsc{Maximum Independent Set} in $r$-locally chordal graphs? 
\end{question}

There are decent approximation algorithms for MIS in high-girth graphs, e.g. \cite{murphy}. We hope that these algorithms can be modified to give approximation algorithms in locally chordal graphs given their $r$-acyclic graph-decompositions $(H, \cV)$ into cliques, by using the high-girth model graph $H$ (on which the high-girth algorithms can be leveraged) and the fact that all bags are cliques (which intersect independent sets in at most one vertex). 

\subsection{Stronger Helly properties}

Let $H$ be a graph and let $\bH$ be a family of subgraphs of $H$. The family $\bH$ has the \defn{Helly-Erd\H{o}s-P\'osa property} if, for every $k \geq 2$, every finite subfamily $\bH'$ of $\bH$ contains either $k$ pairwise disjoint sets or there is a vertex-subset of $H$ of size at most $k-1$ that meets every subgraph of $\bH'$. Observe that when $k=2$ this corresponds to the {\em Helly property} defined in this paper, so the Helly-Erd\H{o}s-P\'osa property can be viewed as a generalization of the Helly property. 

It is known that if $T$ is a tree and $\bT$ is a family of subtrees of $T$, then $\bT$ satisfies the Helly-Erd\H{o}s-P\'osa property, see \cite{ErdosPosaHellySubtreesOfTree}. Therefore, if $G$ is chordal, then the families corresponding to acyclic region representations of $G$ (i.e.\ if $G$ is connected, the subtree representations of $G$) satisfy the Helly-Erd\H{o}s-P\'osa property. 

For $r$-locally chordal graphs, we proved that the families corresponding to their $r$-acyclic region representations satisfy the Helly property. A natural extension of this result is thus to ask whether locally chordal graphs, like chordal graphs, also satisfy the Helly-Erd\H{o}s-P\'osa property: 
\begin{question}\label{q:k-Helly}
    Let $r \geq 3$ be an integer and $G$ an $r$-locally chordal graph. For an $r$-acyclic region representation $v \mapsto H_v$ of $G$, does the family $(H_v \mid v \in V(G))$ satisfy the Helly-Erd\H{o}s-P\'osa property? 
\end{question}

We thank Piotr Micek for suggesting \cref{q:k-Helly} in private communication. 

\subsection{Local-global width parameters}

Chordal graphs can be used to define treewidth: the \emph{treewidth} of a graph $G$ is (one less than) the minimum clique number of a chordal supergraph of $G$. Can we use $r$-locally chordal graphs in a similar way as inspiration to define some width parameter? 

A large part of the contribution of this paper is studying the various properties and behavior of specific graph-decompositions. We identified four key properties of graph-decompositions: $r$-acyclic, connected-$r$-acyclic, $r$-locally acyclic, and inducing $r$-local separations. From this paper, we know the following relationships between these properties:

\begin{theorem}\label{thm:ImplicationsOfDecompsExhibitingLocalGlobal}
    Let $(H,\cV)$ be a graph-decomposition of a graph $G$, and let $r \geq 1$ be an integer. Then the following statements hold:
    \begin{enumerate}
        \item \label{item:AcyclicToLocallyAcyclic} If $(H,\cV)$ is $r$-acyclic, then $(H,\cV)$ is connected-$r$-acyclic.
        \item \label{item:LocallyAcyclicToConAcyclic} $(H,\cV)$ is connected-$r$-acyclic if and only if $(H,\cV)$ is $r$-locally acyclic.
        \item \label{item:ConAcyclicToLocSep} If $(H,\cV)$ is connected-$r$-acyclic and $r \geq 2$, then $(H,\cV)$ induces $r$-local separations. 
    \end{enumerate}
\end{theorem}
\begin{proof}
    \cref{item:AcyclicToLocallyAcyclic} is a consequence of the definitions. \cref{item:LocallyAcyclicToConAcyclic} is \cref{thm:connectedAcyclicEquivLocallyInducingTDs}. \cref{item:ConAcyclicToLocSep} is \cref{thm:connected-acyclic-IMPLIES-separations}. 
\end{proof}

We can use these properties to introduce three new graph width parameters. Let $G$ be a graph. The \defn{width} of a decomposition $(H, \cV)$ of $G$ is the maximum size of its bags. Now: 

\begin{enumerate}
    \item The \defn{$r$-acyclic width} of $G$, denoted \defn{$\aw(G)$}, is the minimum width of an $r$-acyclic graph-decomposition of $G$. 
    \item The \defn{connected-$r$-acyclic width} of $G$, denoted \defn{$\caw(G)$}, is the minimum width of a connected-$r$-acyclic graph-decomposition of $G$. 
    \item The \defn{$r$-local-separation width} of $G$, denoted \defn{$\lsw(G)$}, is the minimum width of a graph-decomposition of $G$ that induces $r$-local separations. 
\end{enumerate}
We also define the \defn{$r$-local treewidth} of $G$, denoted \defn{$\ltw(G)$}, as the maximum treewidth of an $r/2$-ball of $G$ plus one. Together, we refer to these parameters as the \defn{local-global width parameters}. 

The following relationship between the local-global width parameters is immediate from the definitions. 

\begin{theorem}
For every graph $G$ and $r \geq 3$, we have: 
$$\tw(G) \geq \aw(G) \geq \caw(G) \geq \lsw(G), \ltw(G).$$
\end{theorem}

Studying the behavior of these width parameters on different graph classes and for different values of $r$ presents an intriguing research direction. To our knowledge, the local-global width parameters are the first width parameter to be based on graph-decompositions more general than tree-decompositions. This is a promising development: since graph-decompositions are structurally different than tree-decompositions, we can go beyond the tree-based width parameters (which are often closely related to treewidth), but graph-decompositions still offer a decomposition framework, which is useful for studying graph structure. 

The local-global width parameters in particular have the potential to interact in interesting ways with the local complexity of a graph. For example, to what extent is the following statement true: the local-global width parameters are large for a given graph class if and only if the graphs in that class have dense local structure? We can also ask about how the local-global width parameters interact with other measures of the local complexity of graph structure. A good starting place for any of these ideas is the following: 

\begin{question}
    Which natural graph classes have bounded (or unbounded) local-global width parameters?
\end{question}
\newpage

\printbibliography

\end{document}